\documentclass{pspum-l}
\usepackage{amsmath, palatino, mathpazo, amsfonts, amssymb}
\usepackage[mathscr]{eucal}
\usepackage[all]{xy}
\usepackage{mathdots}
\usepackage{datetime}
\usepackage{marginnote}

\newtheorem{theorem}{Theorem}[section]
\newtheorem{lemma}[theorem]{Lemma}
\newtheorem{proposition}[theorem]{Proposition}
\newtheorem{corollary}[theorem]{Corollary}

\theoremstyle{remark}
\newtheorem{remark}[theorem]{Remark}
\newtheorem{example}[theorem]{Example}
\newtheorem{definition}[theorem]{Definition}

\numberwithin{equation}{section}

\newcommand{\M}{\mathcal{M}}
\newcommand{\T}{\Phi}
\newcommand{\Tp}{\Psi}
\newcommand{\bT}{{\bf T}}

\newcommand{\bt}{{\bf t}}
\newcommand{\bK}{{\bf K}}
\newcommand{\bS}{{\bf S}}
\newcommand{\bs}{{\bf s}}

\newcommand{\so}{\mathscr{O}}
\newcommand{\zp}{z \partial}
\newcommand{\p}{\partial}

\newcommand{\f}{\mathbf{f}}
\newcommand{\h}{{\tilde h}}
\newcommand{\La}{\langle\!\langle}
\newcommand{\Ra}{\rangle\!\rangle}
\newcommand{\Lb}{(\! (}
\newcommand{\Rb}{)\! )}

\newenvironment{dedication}
  {
   \thispagestyle{empty}
   \vspace*{\stretch{1}}
   \itshape             
   \raggedleft          
  }
  {\par 
  \vspace{\stretch{3}} 
  }

\begin{document}
\begin{dedication}
In memory of Boris Anatolievich Dubrovin  
\end{dedication}

\title{Quantum flips I: Local model}

\author[Y.-P.~Lee]{Yuan-Pin~Lee}
\address{Y.-P.~Lee: Institute of Mathematics, Academia Sinica, Taipei 10617, Taiwan;
Department of Mathematics, University of Utah, Salt Lake City, Utah 84112-0090, U.S.A.}
\email{yplee.math@gmail.com; yplee@math.utah.edu}

\author[H.-W.~Lin]{Hui-Wen~Lin}
\address{H.-W.~Lin: Department of Mathematics and Taida
Institute for Mathematical Sciences (TIMS),
National Taiwan University, Taipei 10617, Taiwan}
\email{linhw@math.ntu.edu.tw}

\author[C.-L.~Wang]{Chin-Lung~Wang}
\address{C.-L.~Wang: Department of Mathematics and Taida
Institute for Mathematical Sciences (TIMS),
National Taiwan University, Taipei 10617, Taiwan}
\email{dragon@math.ntu.edu.tw}

\date{10 May 2016, revised 30 September 2019}


\maketitle

\begin{abstract}
We study analytic continuations of quantum cohomology under simple flips $f: X \dasharrow X'$ along the extremal ray quantum variable $q^\ell$.
The inverse correspondence $\Tp = [\Gamma_f]^*$ by the graph closure gives an embedding of Chow motives $[\hat{X}'] \hookrightarrow [\hat{X}]$ which preserves the Poincar\'e pairing. 
We construct a deformation $\widehat{\Tp}$ of $\Tp = [\Gamma_f]^*$ which induces a non-linear embedding $$QH(X') \hookrightarrow QH(X)$$ in the category of \emph{$F$-manifolds} into the \emph{regular integrable loci} of $QH(X)$ near $q^\ell = \infty$. This provides examples of \emph{functoriality of quantum cohomology} beyond $K$-equivalent transformations.
In this paper, we focus on the case when $X$ and $X'$ are (projective) local models.
\end{abstract}

\tableofcontents

\setcounter{section}{-1}

\section{Introduction} \label{s:0}

The theory of quantum cohomology emerged about three decades from the study of physics on Calabi--Yau 3-folds. The mathematical foundations have been established and many essential tools including localization techniques and degeneration formulas were developed. These lead to fruitful results on explicit computations of quantum cohomology and enumerative geometry. However, one of the very basic property of usual cohomology, the functoriality under natural morphisms, is still generally lacking for quantum cohomology. 

\subsection{Structure of quantum cohomology} \label{ss:str-QH}

Let $X$ be a complex projective manifold and $\overline{M}_{n}(X, \beta)$ the moduli space of stable maps from $n$-pointed rational nodal curves to $X$ with image class $\beta \in NE(X)$, the Mori cone of effective one cycles. For $i \in [1, n]$, let $e_i: \overline{M}_{n}(X, \beta) \to X$ be the evaluation map. The $g = 0$ Gromov--Witten potential
\begin{equation*}
\begin{split}
F (\bt) = \La - \Ra (\bt) := \sum_{n, \beta} \frac{q^\beta}{n!} \langle
\bt^{\otimes n} \rangle^X_{n, \beta} = \sum_{n \ge 0,\, \beta \in NE(X)}
\frac{q^\beta}{n!} \int_{[\overline{M}_{n}(X, \beta)]^{vir}}
\prod_{i = 1}^n e_i^* \bt
\end{split}
\end{equation*}
is a formal function in $\bt \in H = H(X)$ and the Novikov variables $q^\beta$'s. We call $\mathscr{R} := \Bbb C [\![q^\bullet]\!]$ the (formal) K\"ahler moduli and denote $H_\mathscr{R} = H \otimes \mathscr{R}$. 

Let $\{T_{\mu}\}$ be a basis of $H$ and $\{T^{\mu} := \sum g^{\mu \nu} T_{\nu}\}$
the dual basis with respect to the \emph{Poincar\'e pairing} $g_{\mu \nu} = (T_{\mu}.T_{\nu})$. Denote $\bt = \sum t^{\mu} T_{\mu}$. The \emph{big quantum ring} $(QH(X), *)$ is a $\bt$-family of rings $Q_\bt H(X) = (T_\bt H_\mathscr{R}, *_\bt)$ defined by
\begin{equation} \label{e:q-prod}
\begin{split}
T_{\mu} *_\bt T_{\nu} &:= \sum_{\epsilon, \kappa} \p_\mu \p_\nu \p_\epsilon F (\bt) g^{\epsilon \kappa} T_{\kappa} = \sum_\kappa \La T_\mu, T_\nu, T^\kappa \Ra (\bt)T_\kappa \\
&= \sum_{{\kappa},\,n \ge 0,\, \beta \in NE(X)} \frac{q^\beta}{n!} \langle T_{\mu}, T_{\nu}, T^{\kappa}, \bt^{\otimes n} \rangle^X_{n + 3, \beta} T_{\kappa}.
\end{split}
\end{equation}
The WDVV associativity equations equip $H_\mathscr{R}$ a structure of
\emph{formal Frobenius manifold} over $\mathscr{R}$. It is equivalent to the flatness of the \emph{Dubrovin connection}
$$
\nabla^z = d - \frac{1}{z} A := d - \frac{1}{z} \sum_\mu dt^\mu \otimes T_\mu *_\bt
$$
on the formal relative tangent bundle $TH_\mathscr{R}$ for all $z  \in \mathbb{C}^\times$. The connection matrix $A_\mu$ for $z \nabla_\mu$ is $z$-free ($= T_\mu *$). This uniquely characterizes the constant frame $\{T_\mu\}$ among all other frames $\{\tilde T_\mu\}$ with $\tilde T_\mu \equiv T_\mu \pmod{\mathscr{R}}$. \smallskip
 
Indeed, let $\mathscr{D}^z$ be the ring of differential operators generated by $z\p_i$ with coefficients in $\Bbb C[z][\![q^\bullet, \bt]\!]$. The \emph{$\mathscr{D}^z$ module} associated to $z\p_i \mapsto z\nabla_i$ is isomorphic to $\mathscr{D}^z J$ generated by the $J$ function: let $\psi$ be the class of cotangent line at the first marked section, then
\begin{equation*}
  J(\bt, z^{-1}) := 1 + \frac{\bt}{z} + \sum_{\beta, n, \mu} \frac{q^{\beta}}{n!}
  T_{\mu} \left\langle \frac{T^{\mu}}{z(z-\psi)}, \bt^{\otimes n} \right\rangle_{n+1, \beta}^X
\end{equation*}
which encodes invariants with one descendent insertion. The topological recursion TRR implies (the quantum differential equation QDE)
$$
z\p_\mu \, z\p_\nu J = \sum_\kappa A_{\mu \nu}^\kappa \,z\p_\kappa J.
$$

In practice, one might be able to find element $I(\bt_1 ) \in \mathscr{D}^z J$ but only along some restricted variables $\bt_1 \in H_1 \subset H$. If it happens that $H_1$ generates $H$ (either in classical product or quantum product), then often one may compute $J(\bt)$ and/or $\nabla^z$ effectively. For a toric manifold $X$, such an $I$ function can be found through the $\Bbb C^\times$-localization data with $\bt_1 \in H^{\le 2}(X)$. 
 
\subsection{Statement of results} 
 
Which parts of the structure are functorial? 

Y.~Ruan raised the problem around 1997. Since then, partial inspiring progresses were made on $K$-equivalent manifolds (crepant birational transformations) where the functoriality is the simplest  possible one, namely the equivalence of quantum cohomology under analytic continuations along the K\"ahler moduli. 

The purpose of this paper is to go beyond the  setting of $K$-equivalence and to understand the categorical framework to formulate the functoriality. The simplest non $K$-equivalent birational maps ``preserving the K\"ahler moduli'' are smooth ordinary flips. 

A birational map $f: X \dasharrow X'$ is called a \emph{simple $(r, r')$ flips}, with $r  > r' > 0$, if the exceptional loci of $f$
\[
 Z  \cong P^{r} \subset X, \qquad Z'  \cong P^{r'} \subset X'
 \]
 have the local properties:
\[
 N_{Z/X} \cong \so_{P^r} (-1)^{r'+1}, \qquad N_{Z' / X'} \cong \so_{P^{r'}} (-1)^{r+1}.
\]
The flip $f$ is achieved by blowing up $X$ along $Z$ to get $Y = {\rm Bl}_Z X$, with exceptional divisor $E \cong P^r \times P^{r'}$, and then by contracting $E$ along $P^r$ to get $X'$. 
In this first instalment of this project, we are mainly concerned with the (projective) \emph{local models} of the simple $(r,r')$ flips.
That is, 
\[
 X_{loc} = P_{P^r} \left(\so(-1)^{r'+1} \oplus \so \right), \qquad X'_{loc} = P_{P^{r'}} \left( \so(-1)^{r+1} \oplus \so \right).
\]
In particular, $\dim X_{loc} =\dim X'_{loc} = r+ r' +1$ and $X_{loc}$ is Fano.

It was shown in \cite{LLW1} that the graph closure of $f^{-1}$ defines a correspondence $\Tp = [\bar \Gamma_f]^*$ which identifies the Chow motive of $X'$ as a sub-Chow motive of $X$. While $\Tp: H(X') \to H(X)$ \emph{preserves the Poincar\'e pairing}, it does \emph{not preserve the classical ring (cup product) structure}; see \cite[\S~2.3]{LLW1}. 

The simple flips allow two limits.
When $r=r'$, this is a simple flop.
It was shown in \cite{LLW1} that $\T = [\Gamma_f]_*$ induces an isomorphism $H(X) \cong H(X')$ as vector spaces with bilinear pairing. and the ``anomaly'' of $\T$ with respect to the ring structure is cancelled by Gromov--Witten invariants associated to the extremal rays $\ell \subset X$ and $\ell' \subset X'$. This is understood as analytic continuations along the K\"ahler moduli. In fact the \emph{big quantum cohomology rings} are isomorphic $QH(X) \cong QH(X')$ under analytic continuations induced by $\T$. 
The other limit is $r'=0$ and this is the case of blowing down.
Most of our discussions and results apply to these two limiting cases.

In the case of flips we will show that $QH(X')$ can still be regarded as a sub-theory of $QH(X)$ in a canonical, though \emph{non-linear}, manner.

First of all, there is a basic split exact sequence (cf.~Lemma \ref{l:1})
$$
0 \to K \longrightarrow H(X) \mathop{\longrightarrow}^{\T} H(X') \to 0
$$
with splitting map $\Tp: H(X') \to H(X)$. The kernel space (vanishing cycles) $K$ has dimension $d:= r - r'$ and is orthogonal to $\Tp H(X')$.

Secondly, the Dubrovin connection can be analytic continued ``along the K\"ahler moduli'' to a connection $\T \nabla$ under the rule $\T q^\beta = q^{\T \beta}$ with $\beta \in NE(X)$. As $\T \beta$ might not be effective, indeed $\T \ell = -\ell'$, analytic continuations are generally required. By the very construction, 
$$
\nabla_\mu = \p_\mu - \frac{1}{z} T_\mu * 
$$ 
has only (formal) \emph{regular singularities} at $q_i = 0$ in K\"ahler moduli via the standard identification of \emph{divisorial coordinates} $t^i$ and Novikov variables $q_i$ (which follows from the divisor axiom in GW invariants):
$$
q_i = e^{t^i}, \qquad \p_{i} = q_i \frac{\p}{\p {q_i}}.
$$ 
The resulting connection $\T \nabla$ turns out to be analytic in the extremal ray variable $q^\ell$ and contains \emph{irregular directions} along the divisor $q^\ell = \infty$, that is $q^{\ell'} = 0$, corresponding precisely to the kernel subspace $K$. 

This suggests strongly the possibility of extracting the Dubrovin connection $\nabla'$ on $T H'_{\mathscr{R}'}$, where $H' = H(X')$ and $\mathscr{R}' = \Bbb C [\![NE(X')]\!]$, from $\T \nabla$ by ``removing the $K$ directions''--since after all $\nabla'$ is expected to be regular. 

Indeed, in the next step, it is shown that there is an \emph{eigen-decomposition} 
\begin{equation} \label{e:decomp}
TH \otimes \mathscr{R}'[1/q^{\ell'}] = \mathscr{T} \oplus \mathscr{K}
\end{equation}
into irregular eigenbundle $\mathscr{K}$ which extends $K$ over $\mathscr{R}'[1/q^{\ell'}]$ and the regular eigenbundle $\mathscr{T} = \mathscr{K}^\perp$ which is precisely the orthogonal complement of $\mathscr{K}$. From WDVV equations, both $\mathscr{T}$ and $\mathscr{K}$ are shown to be integrable distributions (cf.~Proposition \ref{p:int-mfd}). The integrable submanifold $\M_{q'}$ passing through the section $(q' \ne 0, \bt = 0)$ is then the proposed manifold corresponding to $QH(X')$. 

The decomposition \eqref{e:decomp} has the flavor of Magrange's theorem on formal decomposition of meromorphic connections. Unfortunately $\T \nabla$ turns out has essential singularities along $q^{\ell'} = 0$ in the naive way. So in practice we start with the \emph{small quantum cohomology} $Q_0 H(X)$ and establish \eqref{e:decomp} in that case first, since $\T \nabla$ is then meromorphic of Poincar\'e rank one along $q^{\ell'} = 0$ (cf.~Lemma \ref{l:K}, \ref{l:C2}). \footnote{Another way is to utilize the ``adic'' topology given by the Mori cone near $q^{\ell'} = 0$: modulo any $q^{\beta}$, the irregularity is of finite order. We do not take this approach here.}

If one now restricts to the local models, the \emph{Picard--Fuchs equations} arising from $\Bbb C^\times$-localizations become available. It turns out that $X$ and $X'$ share the same Picard--Fuchs equations after analytic continuations, and this forms the initial step to compare $\mathscr{T}$ and $QH(X')$. 
Technically there are non-trivial \emph{Birkhoff factorizations} and \emph{generalized mirror transforms} involved to go from Picard--Fuchs equations to Dubrovin connections. 
Still, at the end the functoriality turns out to be quite satisfactory: the product structure can be preserved by deforming the embedding $\Tp$ along the underlying Frobenius manifold, if one is willing to give up the conservation of bilinear pairing. This is known as the $F$-structure. 

\begin{theorem} \label{t:main}
For the local model $f: X \dasharrow X'$ of simple $(r, r')$ flips, there is an $\mathscr{R}'$-point $\sigma_0(q') \in H'_{\mathscr{R}'}$ and an embedding $\widehat \Tp(q', \bs)$ over $\mathscr{R}'$: 
$$
\sigma_0(q') + \bs \mapsto  \widehat \Tp(q', \bs): H(X')_{\mathscr{R}'} \longrightarrow \M \hookrightarrow H(X)_{\mathscr{R}'},
$$ 
where $\bs \in H(X')$, such that 
\begin{itemize}
\item[(1)] 
$(\widehat{\Tp}, \sigma_0)$ restricts to $(\Tp: H' \hookrightarrow H, 0)$ when modulo $q^{\ell'}$, 
\item[(2)] outside the divisor $q^{\ell'} = 0$, the big quantum products on the corresponding tangent spaces are preserved (i.e.~$\widehat{\Tp}$ is an $F$-embedding):
$$
\La \widehat{\Tp}_\mu, \widehat{\Tp}^i, \widehat{\Tp}_j \Ra^X (\widehat{\Tp}(q', \bs)) = \La T'_\mu, T'^i, T'_j \Ra^{X'} (\sigma_0(q') + \bs).
$$
\end{itemize}

In particular, there is a ring decomposition
$$
Q_{\widehat\Tp (q', \bs)} H(X)\cong Q_{\sigma_0(q') + \bs} H(X')\times \Bbb C^{r - r'}.
$$
\end{theorem}

This can be described in more geometric, but perhaps less precise, terms.
One first identify $q^{\ell} = 1/ q^{\ell'}$, and this extend the $A^1_{q^{\ell}}$ to $P^1_{q^{\ell}}$.
Note that $QH(X)$ over the Novikov ring $\mathcal{R}$ is analytic with respect to the variable $q^{\ell}$ and in general formal with respect to other Novikov variables.
One can regard $QH(X)_{\mathcal{R}}$ as an analytic family of Frobenius manifold over $A^1_{q^{\ell}}$.
Near $q^{\ell} = \infty$, Theorem~\ref{t:main} states that there is a family of codimension $d=r-r'$ integrable submanifolds which extends to $q^{\ell} =\infty$ as a family of $F$-manifolds.
Furthermore, in a neighborhood of $q^{\ell} = \infty$, this family of $F$-manifold is isomorphic to the $F$-manifold given by $QH(X')$. 

To give a brief sketch of the proof to Theorem \ref{t:main}, we note that a more precise, and slightly stronger, statement is that $\widehat{\Tp}$ induces an affine (but not Frobenius) embedding over $\mathscr{R}'[1/q^{\ell'}]$:
\begin{equation} \label{e:affine}
\xymatrix{
(T H'_{\mathscr{R}'[1/q^{\ell'}]}, \nabla') \ar@{^{(}->}[r]^{d\widehat{\Tp}}& (T H_{\mathscr{R}'[1/q^{\ell'}]}, \nabla)|_\M \ar[r]& \mathscr{K} \cong N_{\widehat \Tp}}.
\end{equation}
A key step is to prove \eqref{e:affine} in the special case $\bs = 0$ (cf.~Proposition \ref{p:bbgm}), which implies Theorem \ref{t:main} for $\bs = 0$ (cf.~Theorem \ref{t:small}). For general $\bs \in H(X')$, we make use of the local model assumption to get semi-simplicity of $\mathcal{M}$ and $QH(X')$, and then construct the map $\hat\Psi$ by matching the corresponding canonical coordinates (cf.~\eqref{e:Psi}). \footnote{By Proposition \ref{p:bbgm}, the $\mathscr{R}'$-point $\sigma_0(q')$, and hence the embedding $\widehat \Tp(q', \bs)$, is unique if we impose also equation \eqref{e:affine} at $\bs = 0$ in the statement of Theorem \ref{t:main}. Nevertheless we expect that the uniqueness should hold without this additional constraint. } 

While the main results are formulated for the local models, we keep our presentation in theoretic terms whenever possible, with an eye towards future results for global flips.
Indeed, a large part of the proofs works for more general flips with explicit local structures (e.g.~toric flips). Nevertheless all of our results are essentially constructive and effective. The explicit frame leading to Birkhoff factorization and the exact form of connection matrices for the Dubrovin connection for $Q_0 H(X)$ is given. The explicit algorithm for block-diagonalization leading to the eigenbundle decomposition is also given.

As an example to illustrate the ideas involved in the proof of the main theorem, we give the computational details for the $(2, 1)$ flip in the last section (\S \ref{s:(2,1)}, notably Theorem \ref{t:frame}, Lemma \ref{l:extr}, Corollary \ref{c:GMT}, Theorem \ref{t:flip-extr}).

\subsection*{Acknowledgements}
The essential part of this paper was done in 2015--2016 when Y.-P.\ visited H.-W.\ and C.-L.\ at Taida Institute of Mathematical Sciences (TIMS). We are grateful to TIMS for providing excellent working environment to make this collaboration possible. We thank also H.~Iritani for sending us his preprint \cite{hI2} on related results.

\section{From Picard--Fuchs to small $\mathscr{D}^z$-modules}

In this section we study projective local models of $(r, r')$ flips. The classical aspect on cohomology is discussed in \S \ref{s:1.1}. The basic properties of small quantum cohomology are discussed in \S \ref{ss:PF} (Picard--Fuchs ideals) and \S \ref{ss:1st} (first order PDE system).

\subsection{Classical cohomology and correspondence} \label{s:1.1}
We have
\[
 X = P_{P^r} (\so (-1)^{r' + 1} \oplus \so), \qquad X' = P_{P^{r'}} (\so (-1)^{r + 1} \oplus \so).
\]
By Leray--Hirsch, the cohomology ring of $X$ has the following presentation
\[
H(X) = {\mathbb{Z}[h, \xi]}/{(h^{r + 1}, \xi (\xi - h)^{r' + 1})},
\]
where $h$ ( resp. $\xi $) is the hyperplane class of $P^r$ (resp. $X \to P^r$). $H(X)$ has rank $R = (r + 1)(r' + 2)$ with $\mathbb{Z}$-basis
\begin{equation} \label{e:1}
h^{i} (\xi - h)^{j}, \qquad i \in [0, r],\, j \in [0, r' + 1].
\end{equation}
Such a presentation of basis is called a \emph{canonical presentation}. 

Notice that $[Z] = (\xi - h)^{r' + 1}$, and for $i \in [0, r]$
\begin{equation} \label{e:kappa}
k_i := (h|_Z)^i = h^i (\xi - h)^{r' + 1} = (-1)^i (\xi - h)^{r' + i + 1}
\end{equation}
is the class of codimension $i$ linear subspace in $Z$. 

Similar description holds for $X'$ by switching the roles of $r$ and $r'$:
\[
H(X') = {\mathbb{Z}[h', \xi']}/{(h^{r' + 1}, \xi' (\xi' - h')^{r + 1})},
\]
which has rank $R' = (r' + 1)(r + 2)$. We denote a canonical basis by
\begin{equation} \label{e:basis'}
T'_{(i, j)} := (\xi' - h')^{i} h'^{j}, \qquad i \in [0, r + 1], \, j \in [0, r'].
\end{equation} 

It was shown in \cite[\S 2.3]{LLW1} that the Chow motive of $X'$ is a sub-motive of that of $X$ by the correspondence $\Tp = \Phi^* = [\bar\Gamma_{f^{-1}}]$ from $X'$ to $X$, where $\Phi = [\bar \Gamma_f]$. Moreover $\Tp$ preserves the Poincar\'e pairing. Also $\T h = \xi' - h'$, $\T\xi = \xi'$, and $\T$ restricts to an isomorphism on the ideal $\xi.H(X)$ with inverse $\Tp$. The following lemma summarizes it using the basis elements in \eqref{e:1} and \eqref{e:kappa}. Denote
\begin{equation}
d := R - R' = r - r'.
\end{equation}

\begin{lemma} \label{l:1}
The kernel $K$ of $\T: H(X) \to H(X')$ is a free abelian group of rank $d$, generated by
\[
k_i \quad \mbox{with} \quad i \in [0, d - 1].
\]

The image of $\Tp$ in degree $j$ is the full $H^{2j}(X)$ if $j \not\in [r' + 1, r]$. For $j \ge r' + 1$, the image in $H^{2j}(X)$ has a basis given by
\begin{equation} \label{e:basis}
T_{(j - i, i)} := \Tp ( T'_{(j - i, i)} ) = h^{j - i} (\xi - h)^i + (-1)^{r' - i} k_{j - (r' + 1)},
\end{equation}
where $i \in [0, r']$, and the first term vanishes if $j - i \ge r + 1$. 

The pair $(\T, \Tp)$ leads to an orthogonal splitting of $H(X)$:
\begin{equation}
\xymatrix{
0 \ar[r] & K \ar[r] & H(X) \ar[r]^\T & H(X') \ar@/^/[l]^{\Tp} \ar[r] & 0}.
\end{equation}
\end{lemma}

\begin{proof}
It is clear that $K$ has a basis given by $\kappa_i$ with dimension $r - i \ge r' + 1$. 

For $j \ge r' + 1$ and $i \in [0, r']$, we compute
\begin{equation*}
\begin{split}
\Tp (T'_{(j - i, i)})&= \Tp ( (\xi' -h')^{j-i} h'^{i} ) \\
&= \Tp( (\xi'^{j - i}- \cdots +(-1)^{j - i - 1}C^{d-i}_{j - i - 1}\xi' h'^{j - i - 1} +(-1)^{j - i}h'^{j - i}) h'^i)\\
&= \Tp( \xi'^{j - i}h'^i- \cdots +(-1)^{r' - i}C^{j - i}_{r'-i}\xi'^{j - r'}h'^{r'} )\\
&=  \xi^{j - i}(\xi-h)^i - \cdots + (-1)^{r'-i}C^{j - i}_{r'-i}\xi^{j - r'}(\xi-h)^{r'} \\
&=  \xi^{j - i}(\xi-h)^i - \cdots + (-1)^{r'-i}C^{j - i}_{r'-i}\xi^{j - r'}(\xi-h)^{r'}+\\
 & \qquad \qquad \cdots +(-1)^{j - i - 1}C^{j - i}_{j - i - 1}\xi(\xi-h)^{j - 1}\\
&=  (\xi-h)^i(\xi-(\xi-h))^{j - i} - (-1)^{j - i}(\xi-h)^j\\
&=  (\xi-h)^i h^{j - i} - (-1)^{j - i}(-1)^{j - r' - 1}h^{j - r' - 1}(\xi-h)^{r' + 1}\\
&=  h^{j - i} (\xi -h)^i + (-1)^{r' - i} k_{j - (r' + 1)}.
\end{split}
\end{equation*}
If $j \in [r' + 1, r]$, this is orthogonal to $k_{r - j} = h^{r - j}(\xi - h)^{r' + 1}$ by \eqref{e:kappa}.
\end{proof}

Instead of \eqref{e:1}, we will use the elements in Lemma~\ref{l:1} as our basis.

\subsection{Small quantum $\mathscr{D}$-modules via the Picard--Fuchs systems} \label{ss:PF}

The \emph{small quantum cohomology} $Q_0 H(X) = (T_0 H_\mathscr{R}, *_0)$ encodes 3-point invariants by \eqref{e:q-prod}. The fundamental class axiom and divisor axiom show that for $\bt = \hat \bt = t^0 T_0 + D \in H^{\le 2}(X)$ where $D \in H^2(X)$:
$$
\La T_\mu, T_\mu, T^\kappa \Ra (\hat \bt) = \sum_{\beta \in NE(X)} q^\beta e^{D.\beta} \langle T_\mu, T_\mu, T^\kappa \rangle_{3,\beta}.
$$
Thus we may couple together the Novikov variables and the divisor variables and interpret directional derivatives $\p_D$ as derivatives in $q^\beta$'s. The subspace $H^{\le 2}(X)$ is referred as the \emph{small parameter space} and the product $*_{\hat\bt}$ is equivalent to $*_0$. Often we write $*_{small}$ to denote either one of them. The coupled variables are especially suitable for applying (generalized) mirror theorems arising from localization techniques.

For a simple $(r, r')$ flip $f: X \dasharrow X'$, the local models $X$ and $X'$ are both toric manifolds. The small quantum $\mathscr{D}$-modules for toric manifolds are generated by the $I$ function which encodes localization data on stable map moduli spaces. The genus zero Gromov--Witten theory can then be constructed from this $\mathscr{D}$-module $\mathscr{D}^z I$ via the so called BF/GMT procedure. This will be discussed in the next section. 

For the moment we focus on $\mathscr{D}^z I$ and study the corresponding GKZ differential system. In the case of iterated projective bundles, the GKZ system reduces to the Picard--Fuchs system which can be written down easily.

We start with the $X$ side. Let $D = t^1 h + t^2 \xi$ be the divisor variable, $\ell$ and $\gamma$ be the fiber curves for $Z \to {\rm pt}$ and $X \to Z$ respectively. Denote by 
$$
q_1 =q^{\ell} e^{t^1}, \qquad q_2 = q^{\gamma} e^{t^2}
$$ 
the Novikov variables coupled with the ``small parameters'', and $\partial_i =\partial/\partial t^i$. 

The $I$-function is given by 
\begin{equation} \label{e:IX}
\begin{split}
& I = I^X = e^{D/z} \times \\
 &\sum_{d_1, d_2} q_1^{d_1} q_2^{d_2} 
  \frac{1}{\prod_1^{d_1} (h+mz)^{r+1} \prod_1^{d_2-d_1} (\xi-h +mz)^{r'+1} \prod_1^{d_2} (\xi+mz)}
\end{split}
\end{equation}
where $z$ is a formal parameter, and the middle factor goes up as 
\begin{equation} \label{e:up}
(\xi - h)^{r' + 1} \prod_{m = d_2 - d_1 + 1}^{-1} (\xi - h + mz)^{r' + 1} = (-1)^{(r' + 1)(d_1 - d_2 - 1)} k_0 \prod_{m = 1}^{d_1 - d_2 - 1}(h + mz)^{r' + 1}
\end{equation}
when $d_2 - d_1 < 0$. It is annihilated by the following Picard--Fuchs (box) operators
\begin{equation} \label{e:2}
\begin{split}
\Box_{\ell} &:= (\zp_1)^{r+1} - q_1 (\zp_2 -\zp_1)^{r'+1}, \\
\Box_{\gamma} &:= (\zp_2) (\zp_2 -\zp_1)^{r'+1} - q_2,
\end{split}
\end{equation}
The Novikov variables $q^\beta$'s can now be ignored since there is no convergence issue in dealing with equations \eqref{e:2}. Hence we may treat $(q_1, q_2) \in \Bbb C^2$ as variables and identify $\p_i = q_i \p/\p q_i$. 

Since $r > r'$, the PF system for $X$
\[
 \Box_{\ell} I = 0, \qquad \Box_{\gamma} I = 0
\]
is regular holonomic on $\Bbb C^2$ of rank $R$.

On the $X'$ side we have similar notions of $D' = s^1 h' + s^2 \xi'$, $\ell'$, $\gamma'$, 
$$
q_1' = q^{\ell'} e^{s^1}, \qquad q_2' = q^{\gamma'} e^{s^2}
$$ 
and $\p_{i'} = \p/\p{s^i}$. The $I$-function is given by
\begin{equation} \label{e:IX'}
\begin{split}
&I' = I^{X'} = e^{D'/z} \times \\
 &\sum_{d_1', d_2'} q_1'^{d_1} q_2'^{d_2'} 
  \frac{1}{\prod_1^{d_1'} (h' + mz)^{r' + 1} \prod_1^{d_2' - d_1'} (\xi' - h' + mz)^{r + 1} \prod_1^{d_2'} (\xi' + mz)},
\end{split}
\end{equation}
where a similar rule as in \eqref{e:up} applies to the case $d_2' - d_1' < 0$. 

The Picard--Fuchs operators for $X'$ which annihilates $I'$ are  
\begin{equation} \label{e:2'}
\begin{split}
\Box_{\ell'} &:=  (\zp_{1'})^{r' + 1} - q_1' (\zp_{2'} - z\p_{1'})^{r + 1}, \\
\Box_{\gamma'} &:= (\zp_{2'}) (\zp_{2'} - a\p_{1'})^{r + 1} - q_2',
\end{split}
\end{equation}
and the PF system for $X'$ is
\[
 \Box_{\ell'} I' =0, \qquad \Box_{\gamma'} I' =0.
\]

Since $\T(t^1 h + t^2 \xi) = -t^1 h' + (t^1 + t^2) \xi'$, we have 
$$
s^1 = -t^1 \quad \mbox{and}\quad s^2 = t^1 + t^2.
$$ 
Then $\zp_{1'} = \zp_2 - \zp_1$ and $\zp_{2'} = \zp_2$. Also
\begin{equation} \label{e:2.2}
 \T(q_1) = 1/q'_1, \qquad \T(q_2) = q'_1 q'_2.
\end{equation}

\begin{lemma} \label{l:1.2}
The $\mathscr{D}^z$-module defined by the Picard--Fuchs ideal of $X$ is isomorphic to that of $X'$ over
 $\mathbb{C}[q_1, q_1^{-1}, q_2] \cong \mathbb{C}[q'_1, {q'_1}^{-1}, q'_2] $.
\end{lemma}

\begin{proof}
The operators in \eqref{e:2'} for $X'$, written in the variables on $X$, are  
\begin{equation} \label{e:2''}
\begin{split}
\Box_{\ell'} &:=  (\zp_2 -\zp_1)^{r' + 1} - q_1' (\zp_1)^{r + 1}, \\
\Box_{\gamma'} &:= (\zp_2) (\zp_1)^{r + 1} - q_2'.
\end{split}
\end{equation}
Hence $\Box_{\ell'} = q_1^{-1} \Box_{\ell}$ and $\Box_{\gamma'} = z\p_2 \Box_{\ell} - q_1 \Box_{\gamma}$. 
\end{proof}

However, the behavior of the PF system on $X'$ is bad. Since $r > r'$, the expression of $\Box_{\ell'}$ in \eqref{e:2'} shows that it has an irregular singularity at $q_1' = 0$. This is also reflected by the analytic behavior of the $I$ functions:

\begin{lemma}
On the $X$ side, the function $I^X$ is an entire function in $q_1$, while on the $X'$ side the function $I^{X'}$ is divergent, hence only formal, in $q'_1$.
\end{lemma}
\begin{proof}
The convergence radii can be easily deduced from the explicit formulae above. For $X$, \eqref{e:IX} shows that when $d_1$ is large (with $d_2$ fixed) there is a $(d_1)!$ factor appearing in the denominator of the coefficient of $q_1^{d_1}$. On the other hand, for $X'$ with large $d_1'$ the $(d_1')!$ factor appears in the numerator. The lemma then follows from the ratio test of convergence.
\end{proof}

\begin{remark} \label{r:formal}
In principle we may still go from the PF system for $X'$ to get a \emph{formally regular system} with coefficients in formal series by working on the completion of $\Bbb C [\![NE(X')]\!]$ and by applying \eqref{e:2'} inductively. 
\end{remark}

\subsection{The first order linear system} \label{ss:1st}

In general, a \emph{quantized version} of the basis given in Lemma~\ref{l:1} allows us to rewrite the higher order PDEs (PF system on $X$) in terms of systems of first order PDE's (cf.~\cite{Guest, LLW-II})
\begin{equation} \label{e:4}
\begin{split}
 (\zp_i - C_i (q_1, q_2, z))S = 0, \qquad i = 1, 2.
\end{split}
\end{equation}
such that the matrices $C_i$'s are power-series in $q_1, q_2$ and $z$. Here $S$ is the $R \times R$ fundamental solution matrix. In the one variable case this is the standard process to transform an $n$-th order scalar ODE to a first order system.

In the current local case, we have $c_1(X) = (r - r') h + (r'+2)\xi$ and $X$ is a Fano manifold. Then $C_i$'s are indeed polynomials in $q_1, q_2$ and $z$. 

\begin{remark}
From \eqref{e:IX}, and \eqref{e:up}, the $z$ degree for $\beta = d_1 \ell + d_2 \gamma \in NE(X)$ is given by
$$
\begin{cases}
-((r - r') d_1 + (r'+2)d_2) & \mbox{if $d_2 \ge d_1$},\\ -((r - r') d_1 + (r'+2)d_2) - (r' + 1) & \mbox{if $d_2 < d_1$},
\end{cases}
$$
which is $\le -3$ in all cases. Hence $I = J_{small}$, the small $J$ function on $X$, and no mirror transform is needed. However, for the full matrix system \eqref{e:4} there could still be non-trivial $z$-dependence if the frame (quantized basis) is chosen incorrectly. 
\end{remark}

For $X'$, $c_1(X') = -(r - r') h' + (r+2)\xi'$ contains both positive and negative directions and the situation is necessarily complicated as explained in Remark \ref{r:formal}. 

The precise determination of $C_i$'s will be achieved in the next sections. Here we list only the basic properties of them on $X$.

\begin{lemma} \label{l:1.5}
As a scalar function in $q_1$ and $q_2$, each entry of $C_i$ is sub-linear.
Indeed, the entries of $C_i$ can be written as linear combinations of $1$, $q_i$ and $q_1 q_2$.
\end{lemma}

\begin{proof}
The only time $q_i$ occurs is when one uses the $i$-th equation in \eqref{e:2}. The first relation, which involves $(\zp_1)^{r + 1}$, can be used at most once. For $q_2$, the worst case is when one uses the first equation in computing
$$
z\p_1 \Big( (z\p_1)^r (z\p_2 - z\p_1)^j I\Big), \qquad j \in [0, r' + 1],
$$
which gives a factor of $q_1 (\zp_2-\zp_1)^{r' + 1}$ to the right of $(\zp_2 - \zp_1)^j$.
One moment's thought concludes that each final resulting monomial can be at most linear in $q_1 q_2$.
\end{proof}

\begin{corollary} 
The  system \eqref{e:4} is regular singular at $q_1=0$ and irregular singular of Poincar\'e rank 1 at $q_1 = \infty$. It is ordinary at any other value of $q_1$.
\end{corollary}

\begin{proof}
Note that $\partial_i = q_i \partial_{q_i}$. Therefore \eqref{e:4} is regular singular at $q_1 = 0$. From \eqref{e:2.2}, $\T (q_1 q_2) = q_2'$ and the (sub)-linearity guarantee that it is no worse than rank 1 irregular singularity at $q_1' = 0$ (that is $q_1 = \infty$).
\end{proof}

Note that 
\begin{equation} \label{e:5}
 \deg q_1 = r - r' = d, \quad \deg q_2 = r' + 2, \quad \deg z =1.
\end{equation}
Each entry of $C_i$ is then homogeneous in the following sense. Consider the $(i, j)$-th entry of $C_k$. Let $d_i$ be the degree of the $i$-th basis element.
Then 
\begin{lemma} [Homogeneity] \label{l:weight}
\[
 \deg (C_k)_{ij} = d_j - d_i + 1.
\]
Consequently, the highest degree of any entry is $r+r'+2$. In fact only the $(1,R)$-th one has this degree.
\end{lemma}

\begin{proof}
$\zp_k$ increases degree by $1$ and $(C_k)_{ij}$ sends the $j$-th element to the $i$-th element.
\end{proof}

\begin{definition} [Hopf--M\"obius stripe]
By \eqref{e:2.2}, the parameter space $M$ where $(q_1,q_2)$ lies is identified with the total space of $\so_{P^1}(-1)$, which will be called the \emph{Hopf--M\"obius strip}, or the \emph{$q^\ell$-compactified K\"ahler moduli}. 
\end{definition}

We rephrase Lemma~\ref{l:1.5} as follows.

\begin{corollary} \label{c:1.8}
The Picard--Fuchs system defines a meromorphic connection, with parameter $z$, on a trivial rank $R$ vector bundle over the Hopf--M\"obius stripe $M = \so_{P^1}(-1) \to P^1$, with $q_1 = 1/q'_1$ being the coordinate of the base $P^1$.

The connection is regular singular along the divisor $q_1=0$ and irregular singular of Poincar\'e rank 1 along $q'_1=0$.
Furthermore, the irregular singularity does not occur in the differentiation in the fiber direction $q_2$.
\end{corollary}



\section{The GW system for $Q_0 H$} \label{s:GW}

The (small) Dubrovin connection of $X$, which is a toric Fano manifold, can be written down directly by choosing the quantum frame carefully (cf.~Definition \ref{d:q-frame}). This gives the Gromov--Witten invariants for two-point primary invariants without starting at the one-point descendent $J$ function. 

Since the explicit form of the Dubrovin connection is not strictly necessary, we choose to work in a slightly more theoretic manner which is precise enough to study the eigenvalue functions of $h*_{small}$ and $\xi*_{small}$ and to identify the bundle directions leading to irregular singularities near $q^\ell = \infty$, namely the kernel space $K$ (cf.~Lemma \ref{l:K}).

\subsection{Abstract structures of $QH$} \label{ss:abs}

In order to deal with Dubrovin connection in a non-constant frame, which is essential in our proof, we recall some standard structures attached to the quantum cohomology rings. 

The following is well-known

\begin{lemma} \label{metrical}
The Dubrovin connection $\nabla^z := d - z^{-1}\sum_i dt^i \otimes T_i*$ is compatible with the polarized (or Hermitian) pairing: for $H$-valued Laurent series $a(z), b(z)$, 
\begin{equation} \label{e:pairing}
\Lb a(z), b(z) \Rb := (a(z), \overline{b(z)}) := (a(z), b(-z)).
\end{equation}
\end{lemma}

\begin{proof}
On one hand
$$
\theta := \p_i (\sum a_j T_j, \sum \bar b_k T_k) = \sum ((\p_i a_j) \bar b_k + a_j (\p_i \bar b_k)) (T_j, T_k)
$$
On the other hand,
$$
(\nabla^z_i a, \bar b) + (a, \overline{\nabla^z_i b}) = \theta + \sum  a_j\bar b_k \Big(\frac{-1}{z} (T_i*T_j, T_k) + \frac{1}{z} (T_j, T_i*T_k) \Big) = \theta,
$$
where the Frobenius property is used.
\end{proof}

The Dubrovin connection $z\nabla^z$ on any constant frame $T_i$'s (cohomology basis) has its 0-th order operator the matrix $A_k$ of $T_k*$, and has the quantum differential equation $z\p_k z\p_j J = \sum_i (A_k)_j^i z\p_i J$. Hence the fundamental solution matrix $S = (S_j) = (S_j^i)$ with $z\p_k S = A_k S$ is determined by the adjoint relation
$$
(\sum\nolimits_i S_j^i T_i, T_k) = (T_j, z\p_k J).
$$ 
That is,
$$
S^i_j = (T_j, \sum\nolimits_k g^{ik} z\p_k J) = \sum\nolimits_{k, l} g^{ik} g_{jl}\, z\p_k J^l.
$$

In terms of a non-constant frame $\tilde T_j = \sum_i T_i\, p_j^i (q, \bt, z)$ as power series in $q, \bt$ and $z$, the corresponding fundamental solution matrix $Z$ satisfies $S = PZ$ with $P = (p_j^i)$, and the equation becomes $z\p_k Z = \tilde A_k(q, \bt, z) Z$ with
\begin{equation} \label{e:tAk}
\tilde A_k = -zP^{-1}\p_k P + P^{-1} A_k P.
\end{equation}

\begin{remark}
Even if $P$, or equivalently $\tilde T_i$'s, is independent of $z$, the connection matrices $\tilde A_k$ might still be $z$-dependent if $P(q, \bt)$ is not constant in $(q, \bt)$. On the other hand, for a (formal) change of variables $(q, \bt) \mapsto (\tilde q, \tilde \bt)$ we get a linear change 
\begin{equation} \label{e:CV}
A_k \mapsto \tilde A_l = \sum_k A_k (\p t^k/\p \tilde t^l)
\end{equation} 
which is $z$-independent if $A_k$'s are. When both operations are performed the connection matrices $\tilde A_k(\tilde q, \tilde \bt, z)$'s are usually complicated.
\begin{itemize}
\item[(1)]
A typical case for this to occur is the connection matrix obtained from the $I$ function. In that case one uses Birkhoff factorization (BF) to recover the frame $\tilde T_j$ to get $z$-independent connection matrices and use generalized mirror transform (GMT) to recover the change of coordinates if any. This is discussed in \S \ref{s:small}.
\item[(2)]
The block diagonalization/decomposition of connections gives another instance of this construction. This is discussed in \S \ref{ss:decomp}.
\end{itemize}
We will study non-constant frames arising from combinations of these two ``gauge transformations''.
\end{remark}

As a linear map, the matrix of $T_k*$ in the basis (non-constant frame) $\tilde T_j$'s is given by $P^{-1} A_k P$ instead of $\tilde A_k$. Hence 
\begin{equation} \label{e:z}
T_k * \tilde T_j = \sum_i \tilde T_i (\tilde A_k)_j^i  +  \sum_i \tilde T_i (P^{-1}\,z\p_k P)_j^i.
\end{equation}
In particular, on the $\deg z = 0$ component we get
\begin{equation} \label{e:z=0}
T_k * \tilde T_j(0) = \sum_i  \tilde T_i(0) \tilde A_{kj}^i (0).
\end{equation}
In terms of GW invariants we have
$$
\tilde A_{kj}^i (0) = \sum_l \tilde g^{i l}(0) \La T_k, \tilde T_j(0), \tilde T_l(0) \Ra = \La T_k, \tilde T_j(0), \tilde T^i(0) \Ra,
$$ 
where 
\begin{equation} \label{e:dual}
\tilde T^{i} :=  \tilde g^{i l}\, \tilde{T}_l
\end{equation} 
is the dual frame with respect to the polarized pairing 
$$
\tilde g_{i j} := \Lb \tilde T_i, \tilde T_j \Rb.
$$ 
The pairing becomes symmetric when we restrict to $z = 0$.


\begin{remark} \label{r:z}
While \eqref{e:z=0} holds for any frame $\{\tilde T_j(q, \bt, z)\}$, a special frame such that $\tilde A_k(q, \bt, z)$ is $z$-independent is of fundamental importance. In fact it is unique up to a constant transformation matching the constant basis $\tilde T_j \pmod{NE(X)}$ with the original one. Nevertheless, as we shall see later (cf.~\eqref{e:gauge}), non-trivial non-constant frames with $z$-independent $\tilde A_k$ do exist when we consider analytic continuations in certain $q$ variables toward infinities. Of course in that case $\tilde T_j$ is not defined near $q = 0$.   
\end{remark}

\subsection{Dubrovin connection on $Q_0 H(X)$}

The system defined by the (small) Dubrovin connection of $X$ is
\begin{equation} \label{e:6}
\begin{split}
 (\zp_i - A_i (q_1, q_2))S &= 0, \qquad i = 1, 2,
\end{split}
\end{equation}
where $A_1 = h *_{small}$ and $A_2 = \xi *_{small}$ are the matrices defined by the (small) quantum product. Notice the characteristic feature that $A_i$'s are independent of $z$ by definition.

\begin{lemma} \label{l:Fano}
For Fano $X$, $A_i$'s are polynomial functions in $q_1$ and $q_2$.
\end{lemma}

\begin{proof}
By \eqref{e:5} we know that degrees of $q_1$ and $q_2$ are both strictly positive. 
Therefore, we have the polynomiality in $q_j$.
\end{proof}

The (small) Dubrovin connection on $X$ extends meromorphically over the parameter space $M$, with regular singularity on the fiber divisor $q_1 = 0$ and possibly irregular singularity on the fiber $q_1' = 0$. A detailed determination is given in this subsection. 

Before doing so, we first describe the eigenvalue functions $\lambda(q_1, q_2)$ of $h*_{small}$ and $\mu(q_1, q_2)$ of $\xi*_{small}$ in terms of the Picard--Fuchs system \eqref{e:2}. It is important to notice that, since $h$ and $\xi$ are of degree one, by definition the eigenvalue functions are also of degree one. 

Since $X$ is toric Fano, no mirror transformation is needed for small quantum cohomology and we get 
\begin{equation} \label{e:eigen}
\lambda^{r + 1} = q_1 (\mu - \lambda)^{r' + 1}, \qquad \mu (\mu - \lambda)^{r' + 1} = q_2.
\end{equation}
Then we clearly have a simple relation
\begin{equation} \label{e:eigen1}
\mu \lambda^{r + 1} = q_1 q_2,
\end{equation}
and we may use \eqref{e:eigen1} to eliminate $\mu$ in \eqref{e:eigen} to get the equation for $\lambda$:
\begin{equation} \label{e:lambda}
\lambda^{(r + 1)(r' + 2)} = q_1 (q_1 q_2 - \lambda^{r + 2})^{r' + 1}.
\end{equation}
Since $(r + 1)(r' + 2) - (r + 2)(r' + 1) = R - R' = r - r' = d > 0$, all the solutions $\lambda(q_1, q_2)$ are analytic in $q_1, q_2$ as expected. 

It is also clear from \eqref{e:eigen1} and \eqref{e:lambda} that the small quantum product on $X$ is generically semi-simple. Since the semi-simplicity is an open condition, we conclude also the generic semi-simplicity for big quantum product. 

\begin{remark} \label{r:toric}
In \cite{Iritani}, Iritani proved that the big quantum cohomology of any smooth projective toric variety is convergent and generically semi-simple.
\end{remark}

However, under the analytic continuation $x = q_1' = 1/q_1$, $y = q_2' = q_1 q_2$ to the locus $x = 0$, equation \eqref{e:lambda} becomes
\begin{equation} \label{e:lambda-x}
\begin{split}
0 &= x\lambda^{(r + 1)(r' + 2)} - (y - \lambda^{r + 2})^{r' + 1} \\
& = x \lambda^R - (-1)^{r' + 1}\lambda^{R'} - \sum\nolimits_{j = 1}^{r' + 1} (-1)^{r' + 1 - j} C^{r' + 1}_j y^j \lambda^{R' - (r + 2)j}.
\end{split}
\end{equation}
The leading terms $\lambda^{R'}(x \lambda^{d} - (-1)^{r' + 1})$ lead to the following. 

\begin{lemma} \label{l:irr-root}
Near $x = 0$, there are $d = r - r'$ singular eigenvalue functions 
$$
\lambda_i(x^{1/d}, y) = \omega^i x^{-1/d} + \ldots\ 
$$ 
of $h*_{small}$, where $\omega^d = (-1)^{r' + 1}$. The corresponding eigenvalue for $\xi *_{small}$ is 
$$
\mu_i(x^{1/d}, y) = \omega^{-(r + 1)j} x^{(r + 1)/d} y + \ldots.
$$
\end{lemma}

We will see that they correspond to the space of vanishing cycles $K$.

\begin{definition} [Naive quantization frame] (cf.~\cite{LLW-II})
We use the notations of naive quantizations when a cohomology class is represented by a product of divisors in a canonical manner. Namely for any divisor $D$ we set $\hat D = z\p_D$ as a directional derivative, and for a class $a = \prod_i D_i^{e_i}$ under the fixed canonical presentation we set $\hat a = \prod_i \hat D_i^{e_i}$ as a higher order derivative. 

It is easy to see that $\hat a e^{D/z} = a e^{D/z}$ where $D = \sum t^i D_i \in H^2(X)$ is a general divisor. In particular $\hat a I \equiv a e^{D/z} \pmod{NE(X)}$.
\end{definition}

\begin{definition} [The $\Tp$-corrected quantization frame] \label{d:q-frame}
The quantized basis corresponding to the kernel of $\T$ is chosen to be the naive ones
\begin{equation}\label{k:1}
\kappa_i :=\hat k_i I = \hat h^i(\hat \xi - \hat h)^{r'+1} I,
\end{equation}
where $i \in [0, d - 1]$.

For classes in the image of $\Psi$, a correction term will be inserted as follows: for $\mathbf{e} = (e_1, e_2)$ with $e_1 \in [0, r + 1]$ and $e_2 \in [0, r']$, we define
\begin{equation} \label{e:3}
\begin{split}
v_{\mathbf{e}} &:=  \hat h^{e_1} (\hat\xi - \hat h)^{e_2} I + {\delta}_{(e_1,\, e_2)}  (-1)^{r'-e_2} \hat k_{e_1 + e_2 - (r' + 1)} I,
\end{split}
\end{equation}
where 
\begin{equation*}
\begin{cases}
\delta_{(e_1,\, e_2)} = 0 &\mbox{if $e_1 + e_2 \le r'$, and} \\
\delta_{(e_1,\, e_2)} =1 &\mbox{otherwise}. 
\end{cases}
\end{equation*}
The frame is called $\Psi$-corrected since \eqref{e:3} is equivalent to 
$$
v_{\mathbf{e}} = \hat T_\mathbf{e} I + \delta_{e_1, r + 1} (\Psi \hat T'_{\mathbf{e}}) I.
$$
%

When modulo $q_1$, $q_2$, this frame $\{v_\mathbf{e}, \kappa_i\}$ reduces to the constant frame $\{T_\mathbf{e}, k_i\}$ which is consistent with the one given in Lemma \ref{l:1}.
\end{definition}

We investigate the structure on the kernel part. It is clear that 
$$
\zp_1\kappa_j  = \kappa_{j + 1} \quad \mbox{for} \quad 0 \leq j \leq d - 2,
$$ 
and by \eqref{e:2},
\begin{equation} \label{q1-appear}
\begin{split}
\zp_1\kappa_{d - 1} &= {\hat h}^{r-r'}(\hat \xi - \hat h)^{r'+1} I\\
&= \big({\hat h}^{r-r'}(\hat \xi - \hat h)^{r'+1} I - (-1)^{r' + 1} {\hat h}^{r + 1} I\big) + (-1)^{r' + 1} {\hat h}^{r + 1} I \\
& = (-1)^{r'}v_{(r+1,\, 0)} + (-1)^{r' + 1}q_1\, \kappa_0.
\end{split}
\end{equation}

Similar calculations lead to the matrices $C_j(q_1,q_2)$, $j = 1, 2$, explicitly. The miracle is that there is no $z$-dependence under the $\Psi$-corrected quantization frame in \eqref{k:1} and \eqref{e:3}, hence we have $A_j = C_j$ for $j = 1, 2$. To be explicit, we write the connection matrices $C_j$, $i = 1, 2$, in the block form with respect to the decomposition $H(X) = \Tp H(X') \oplus^\perp K$:
\[
C_j = 
\begin{bmatrix} C_j^{11} & C_j^{12} \\ C_j^{21} & C_j^{22} \end{bmatrix}.
\]
We emphasize that \eqref{q1-appear} is the only place where the monomial $q_1$ appears in $C_1(q_1, q_2)$. Namely it is the $(R' + 1, R)$-th entry. In all the other entries the non-trivial monomials appeared are $1$, $q_2$ and $q_1 q_2$:

\begin{lemma} \label{l:K}
For $C_1$, the block corresponding to the kernel subspace is given by
\begin{equation} \label{K-block}
C_1^{22} = \begin{bmatrix}
&&& (-1)^{r' + 1} q_1 \\
1\\
&\ddots\\
&& 1
\end{bmatrix},
\end{equation}
where all blank entries are zero. It has characteristic polynomial $\lambda^d - (-1)^{r' + 1} q_1$.

All the other entries in $C_1$ are either $0, 1$, or $q_1 q_2$ up to sign. 

Moreover, for $C_1^{21}$ the constant terms appear only in the first row whose column has degree $r'$. All other entries are zero.
\end{lemma}

\begin{lemma} \label{l:C2}
In $C_2(q_1, q_2)$, the non-trivial entries consist of monomials only. The monomials appeared in the entries are $0$, $1$, $q_2$ and $q_1 q_2$ up to sign. 

Moreover, the only non-zero entries in $C_2^{21}$ are $q_2$ up to sign. 
\end{lemma}

The proof of the remaining part of Lemma \ref{l:K} as well as a complete proof of Lemma \ref{l:C2} are straightforward computations based on the Picard--Fuchs equations \eqref{e:2}, similar to the one in \eqref{q1-appear}. They are written in \S \ref{ss:explicit}. 

Here we emphasize that the difference between $C_1^{21}$ and $C_2^{21}$ on the constant terms is due to the fact that elements in $K$ can not contain the $\xi$ factor.

\section{Decomposition of $Q_0 H$ via block-diagonalization} \label{ss:decomp}

We learned from Lemma \ref{l:K} and Lemma \ref{l:C2} that for simple $(r, r')$ flips (with $r > r'$) the Dubrovin connection is irregular of Poincar\'e rank 1 at $q^\ell = \infty$. Over the Hopf--M\"obius strip $M$, the Dubrovin connection is a system of first order PDE's of two variables $x = q_1' = 1/q_1$ and $y = q_2' = q_1 q_2$. Recall that $R = \operatorname{rank} H(X)$, $R' = \operatorname{rank} H(X')$, and $d = \operatorname{rank} K = r - r'$ (so that $R = R' + d$). In the $q_1'$ direction, it takes the form
\[
z q_1' \frac{\p}{\p q_1'} S = A S,
\]
where $S$ is the fundamental solution matrix and $A$ is the connection matrix of size $R \times R$. $A$ is entire in $q_1$ but has a simple pole at $q_1' = 0$.

The solution of ODE of this type was developed by Sibuya, Malgrange, Wasow etc.\ (cf.~\cite{Wasow, Sibuya}), and completed in early 1970's. One key step is to block-diagonalize the matrix $A$, starting from the worst singularity. It turns out that this ``classical'' procedure produces an ideal of quantum multiplication generated by $K$, which is however NOT an ideal in $H(X)$! By the flatness of the Dubrovin connection we may simultaneously block-diagonalize all quantum multiplication matrices.

\subsection{Block diagonalization} \label{s:5.2}
%

We have $A_j = C_j$, $j = 1, 2$. From Lemma \ref{l:K}, \small
\[
C_1^{22} = \begin{bmatrix} 0 & 0 &\cdots & (-1)^{r' + 1} q_1 \\
1 & 0 & \cdots & 0 \\
&\ddots \\
0 & \cdots &1 &0 
\end{bmatrix}
= \frac{1}{x} \begin{bmatrix} 0 & 0 &\cdots & (-1)^{r' + 1} \\
x & 0 & \cdots & 0 \\
&\ddots \\
0 & \cdots &x &0 
\end{bmatrix}.
\] \normalsize
We will now work on the irregular system of partial differential equations in variables $(x, y)$ with a parameter $z$.
The irregularity comes only from $x$, and it is thus necessary to keep track of the lowest order entries in $x$ in the connection matrix.
By \S \ref{ss:zp1}, the only non-zero row in $C_1^{21}$ where the lowest (constant) order entry occurs comes from the first row given in \eqref{e:edge}. 
For convenience, we \emph{drop the explicit dependence on $y$} from the notations below when no confusion will likely arise.
%

A transformation is needed to bring $C^{22}_1$ into its ``semisimple'' form: let $u = x^{1/d}$, we modify the constant frame in Definition \ref{d:q-frame} to $\{T_i\}$ with
\begin{equation} \label{e:q-sf}
\{T_i\}_{i = 0}^{R' - 1} = \{T_{\bf e}\}, \qquad \{T_{R' + i}\}_{i = 0}^{d - 1} = \{u^{i} k_{i} \}_{i = 0}^{d - 1}.
\end{equation}

\begin{lemma} [Shearing] \label{l:5.1}
Let 
\[
 Y(x) =\operatorname{diag}( 1^{R'}, u^0, u^1, \cdots, u^{d-1}). 
\]
The equation
\begin{equation} \label{e:5.2}
zx \frac{\p}{\p x} S = C_1 S,
\end{equation}
after the substitutions $S = Y W$ and $x=u^d$, becomes
\begin{equation} \label{e:5.3}
  z u \frac{\p}{\p u} W = D_1(u) W,
\end{equation}
where $D_1$ can be written in the block form as
\begin{equation} \label{e:5.4}
 \begin{split}
 D_1^{11} &= d \cdot {C}_1^{11},  \\
 D_1^{12} &= d \cdot {C}_1^{12} \cdot \operatorname{diag} (u^0, u^1, \cdots,  u^{d-1}), \\
 D_1^{21} &= d \cdot \operatorname{diag} (u^0, u^{-1}, \ldots,  u^{-d+1}) \cdot   {C}_1^{21}, \\
  D_1^{22} &= \frac{d}{u} \cdot \begin{bmatrix}
  0 & 0 &\cdots &(-1)^{r'+1} \\
 1 & -z\frac{1}{d} u & \cdots & 0 \\
& \ddots &\ddots \\
  0 & \cdots &1 & -z\frac{d-1}{d} u
   \end{bmatrix} .
 \end{split} 
\end{equation}

Furthermore, $D_1^{21}$ is a power series in $u$.
Thus, \eqref{e:5.3} is irregular of Poincar\'e rank $1$ in $u$,  and the irregular part only appears in the $(2,2)$ block.
\end{lemma}

\begin{proof}
The computation of the sheared connection matrix $D$ is straightforward.
The last assertion about the regularity of $D_1^{21}$ follows from Lemma \ref{l:K} that the constant term of $C_1^{21}$ only appears in the first row, with other entries being zero. This concludes the proof.
\end{proof}

\begin{remark} \label{r:5.2} 
For the equation related to $C_2$, $ z y\, \p_y S = C_2 S$,
we note that $C_2$ is holomorphic in $x$ and $y$. 
After the shearing the equation becomes
\begin{equation} \label{e:Dy}
  z y \frac{\p}{\p y} W = D_2 W,
\end{equation}
such that $D_2^{21} = d \operatorname{diag} (u^0, u^{-1}, \cdots,  u^{-(d - 1)}) {C}_2^{21}$. 
By Lemma \ref{l:C2} the non-trivial entries in $C_2^{21}$ must divide $q_2 = xy$. Therefore, $D_2$ is still holomorphic in $u$ and $y$. 
\end{remark}

We note that the lowest degree term of $D_1$ in $u$ is of the (block) form
\[
\begin{bmatrix}
  0 & 0 \\
  0 & D_1^{22}(0)
 \end{bmatrix}  \frac{du}{u} 
 \]
such that 
\begin{equation} \label{e:D221}
 D_1^{22} (0)= \begin{bmatrix}
0 & 0 & \ldots & (-1)^{r'+1} \\
1 & 0 & \ldots & 0 \\
&\ddots \\
0 & \ldots &1 & 0
\end{bmatrix}.
\end{equation}
Therefore, $D_1(0)$ has $R$ eigenvalues, including $0$ with multiplicity $R'$ 
and $d$ \emph{distinct nonzero} eigenvalues coming from $D_1^{22}(0)$.
The latter group consists of $d$ distinct solutions of $\omega^d = (-1)^{r' + 1}$. (As we have seen in Lemma \ref{l:irr-root}.)

Following the classical procedure as in \cite{Wasow}, together with the flatness of the Dubrovin connection, we conclude that

\begin{proposition} \label{p:5.3}
The connection matrices $C_1$ and $C_2$ can be simultaneously block diagonalized, such that the $(2, 2)$ blocks is completely diagonalized. 

Furthermore, the block-diagonalization frame $\{\tilde T_i\}_{i = 0}^{R - 1}$ can be chosen so that $\tilde T_i$ has the initial term $T_i$ in $u$. Consequently the bundle $\mathscr{T}$ spanned by $\tilde T_i$ with $i \in [0, R' - 1]$ and $\mathscr{K}$ spanned by $\tilde T_j$ with $j \in [R', R - 1]$ are orthogonal to each other.\end{proposition}

\begin{proof}
Since the $(1, 1)$ block and $(2, 2)$ block do not share any eigenvalues,
the block diagonalization is possible.
The complete diagonalization of the $(2, 2)$ block follows from the fact that 
all eigenvalues of $D_1(0)$ are different in the $(2, 2)$ block.

As explained also in Remark~\ref{r:5.2}, we can use the same shearing
transformation matrix for $C_1$ and $C_2$.
We need to simultaneously diagonalize the sheared counterparts (i.e.~the $(2, 2)$ blocks) of $C_1$ and $C_2$. This is doable as they form part of the flat connection. The flatness together with suitable boundary condition makes the process possible.

To be precise, let $\tilde T_i(u, y, z)$ be the frame leading to block diagonalization for $\nabla_1$ such that $\tilde T_i$ has $T_i$ as the initial term. Then for $i \in [0, R' - 1]$, 
$$
\nabla_1 \tilde T_i = \sum_{j = 0}^{R' - 1} e_{1i}^j \tilde T_j
$$
for some power series $e_{1i}^j(u, y, z)$ and $\tilde T_i(0, y, z) = T_i$.

We claim that the sub-bundle $\mathscr{T}$ spanned by $\tilde T_i$, $i \in [0, R' - 1]$ is also closed under $\nabla_2$, i.e.~$\nabla_2 \tilde T_i \in \mathscr{T}$. For the initial value along $u = 0$ we have $\nabla_2 \tilde T_j (0, y, z) = \nabla_2 T_j$. By Lemma \ref{l:C2}, the block $C_2^{21}$ vanishes since $q_2 = xy = 0$ along $u = 0$. Hence $\nabla_2 T_j \in \mathscr{T}$. Now
$$
\nabla_1 (\nabla_2 \tilde T_i) = \nabla_2 \nabla_1 \tilde T_i = \sum_{j - 0}^{R' - 1} (\p_2 e_{1i}^j) \tilde T_j + \sum_{j - 0}^{R' - 1} e_{1i}^j (\nabla_2 \tilde T_j).
$$
The uniqueness theorem of ODE in $u$ then implies that $\nabla_2 \tilde T_i \in \mathscr{T}$.

The bundle $\mathscr{T}^\perp$ under the pairing \eqref{e:pairing} is closed under $\nabla_i$, a fact which follows from Lemma \ref{metrical} easily. Indeed for all $v \in \mathscr{T}$ and $w \in \mathscr{T}^\perp$, we have $0 = \p_i \Lb v, w \Rb = \Lb \nabla_i v, w\Rb + \Lb v, \nabla_i w\Rb = \Lb v, \nabla_i w\Rb$. Hence $\nabla_i w \in \mathscr{T}^\perp$. This in particular implies that $\mathscr{K} = \mathscr{T}^\perp$. 

The proof that $\nabla_2 \tilde T_j$ is proportional to $\tilde T_j$ for $j \in [R', R - 1]$ is similar and thus omitted.
\end{proof}

\begin{remark} \label{r:EV}
The initial term of the frame which leads to the block diagonalization is the starting frame in \eqref{e:q-sf}. The initial terms of the frame further diagonalizes the $(2, 2)$ block corresponds to the eigenvectors of $D^{22}_1(0)$ in \eqref{e:D221} under the starting frame. Let $K_j$ be the eigenvector with eigenvalue $\omega_j$ where $\lambda^d - (-1)^{r' + 1} = \prod_{j = 0}^{d - 1} (\lambda - \omega_j)$. Then it is easy to see that
\begin{equation} \label{e:EV}
K_j = \sum_{i = 0}^{d - 1} \omega_j^{-i} u^i k_i, \qquad j \in [0, d - 1].
\end{equation}
\end{remark}

\begin{proposition} \label{p:5.4}
After the block-diagonalization, the $(1, 1)$ block of equation \eqref{e:5.3}
can be written in terms of $x$, instead of $u =x^{1/d}$, and we get
\[
 z x \frac{\p}{\p x} Z = \tilde C_1^{11} Z, \qquad z y \frac{\p}{\p y} Z = \tilde C_2^{11} Z,
\]
where $\tilde C_j^{11}$'s are power series in $x$, $y$ and $z$. 
\end{proposition}

\begin{proof}
We will concentrate on $\tilde C_1$ where most of the action happens. Then the question is essentially reduced to an ODE in variable $u$, with $y, z$ acting as parameters. In the following the dependence on $y$ is mostly suppressed since it does not participate in the formal process involving $u$ and $D_1$. The steps involved are to apply the algorithm described in \cite[\S11]{Wasow}.

For notational convenience we rewrite \eqref{e:5.3} as follows
\begin{equation} \label{e:5.5}
  zu^2 \frac{\p}{\p u} W = \bar{D}(u, z)  W,
\end{equation}
where 
$$
\bar{D} = \sum_{l=0}^{\infty} \bar{D}_l u^l
$$ 
as a matrix-valued power series in $u$.
In particular, the subscripts now stand for the exponent of power series for the duration of this proof
(and the $1$ and $2$ of $D$ is temporarily suppressed). 
Similarly, let 
$$
P(u) = \sum_{l=0}^{\infty} P_l u^l
$$ 
with $P_0 = \mathrm{I}$ and $P_l$ being off-block-diagonal for $l > 0$. 
Now we perform a gauge transformation 
$$
W = P Z
$$ 
with new frame 
\begin{equation} \label{e:newT}
(\tilde T_0, \ldots, \tilde T_{R - 1}) = (T_0, \ldots, T_{R - 1}) P
\end{equation} 
to equation \eqref{e:5.5}, aiming to get the connection matrix in the block diagonalized form. That is,
\begin{equation} \label{e:5.6}
  zu \frac{\p}{\p u} Z = \frac{E}{u}  Z, \quad \text{with} \quad  E^{12} =0 = E^{21}.
\end{equation}
By writing 
$$
E = \sum_{l = 0}^\infty E_l u^l,
$$ 
this is equivalent to solving the following system of algebraic equations inductively \cite{Wasow}: 
\[
\begin{split}
 {E}_l^{11} &= - H_l^{11},  \\
 {E}_l^{22} &= - H_l^{22}, \\
 P_l^{12} &= - H_l^{12} (\bar{D}_0^{22})^{-1}, \\
 P_l^{21} &= (\bar{D}_0^{22})^{-1} H_l^{21} ,
\end{split}
\]
where
\[
H_l : = \sum_{s = 1}^{l - 1} P_s {E}_{l - s} - \sum_{s = 0}^{l - 1} \bar{D}_{l - s} P_s - z (l - 1) P_{l - 1}
\]
is determined by $P_s$ and ${E}_s$ for $s \leq l - 1$.
Note that $\bar{D}_0^{ij} =0$ unless $(ij) =(22)$.

Now we can use these equations and the facts that the off-diagonal blocks of $\bar{D}$ have the specific form to perform the induction. 
It is enough to show that $H_l^{11}$ is in powers of $x = u^d$ only. 
Note that ${\bar{E}}$ has vanishing off-diagonal blocks and $P$ has vanishing diagonal blocks (except for $P_0 = I$). 
We see from the above that, for $l \geq 1$,
\[
H_l^{11} = - \bar{D}^{11}_l - \sum_{s = 0}^{l - 1} \bar{D}_{l - s}^{12} P_s^{21} = - ( \bar{D}^{11} +\bar{D}^{12} P^{21})_l,
\]
since $E^{21}=0=\bar{D}_0^{12}$ and $P^{11}_{>0} =0$. Now note that 
\[
\begin{split}
u^{-1} \bar{D}^{12} &=  d \cdot {C}_1^{12}(x) \cdot 
     \operatorname{diag} (1, u^{1}, \cdots,  u^{d - 1}), \\
u^{-1}\bar{D}^{21} &=  d \cdot \operatorname{diag} (1, u^{-1}, \cdots,  u^{-(d - 1)})
   \cdot {C}_1^{21}(x). \\
\end{split}
\]
Inductively, it can be shown that 
\[
{P}^{21} =  \operatorname{diag} (1, u^{-1}, \cdots,  u^{-(d - 1)}) . (\text{matrix function in $x$}).
\]
Thus, $\operatorname{diag} (1, u^{1}, \ldots,  u^{d-1})$ is always cancelled
by its inverse in the $(11)$ block.
Since $\bar{D}^{11}/u$ depends only in $x$, we conclude that $\tilde C^{11}_1$ is a power series in $x$ (and $y, z$).

The proof for $C_2$ is simpler and hence omitted.
\end{proof}

\subsection{Decomposition of small quantum rings}

The quantum product $*$ in this subsection is assumed to be the small quantum product on $H(X)$. 

The Dubrovin connection is flat and the connection matrices $C_1$ and $C_2$ are simultaneously block-diagonalized to $\tilde C_1$ and $\tilde C_2$ respectively. Since $h*$ and $(\xi-h)*$ generate the quantum ring, which is commutative, we conclude that $C_1$, $C_2$ generate the matrix $C_\mu$ for $T_\mu *$ and $\tilde C_1$, $\tilde C_2$ induce block-diagonalization of all $C_\mu$'s, i.e.~ the entire small quantum ring, to $\tilde C_\mu$'s. 

Indeed for $a, b \in H(X)$ we have $ab = a*b + \sum_\beta q^\beta c_\beta$ for some $c_\beta \in H(X)$. Hence by induction on the Mori cone we conclude that $h*$ and $\xi*$ generate the small quantum algebra over the Novikov ring. Namely
\begin{equation} \label{e:q-div}
T_\mu * = \sum_{\beta \in NE(X)} q^\beta P_\beta (h*, \xi*)
\end{equation}
where $P_\beta$ is a polynomial. Since $X$ is Fano (cf.~Lemma \ref{l:Fano}), \eqref{e:q-div} is actually a finite sum. The top degree term $P_0$ is the cup product expression for $T_\mu$.

In particular the block diagonalization under variables $u = x^{1/d}, y, z$ extends to all $T_\mu *$. Moreover, it follows from Proposition \ref{p:5.4} that all the corresponding $(1, 1)$ blocks are still expressible in terms of $x, y$ and $z$. 

Nevertheless, two issues needs to be taken care in details:
\begin{itemize}
\item[(i)] Remove the $z$-dependence introduced in the block-diagonalization to interpret the product structure correctly. 
\item[(ii)] Identify the ground ring where the construction works. Since $T_\mu *$ is generated by $h*$ and $\xi*$ over $NE(X)$ instead of over $NE(X')$, the $(1, 1)$ block of $\tilde C_\mu$ might contains negative powers in $x$ even if $T_\mu$ is in the image of $\Psi$.
\end{itemize}

Denote the frame leading to the block diagonalization by
$$
\mathscr{F} = \{ \tilde T_0, \ldots , \tilde T_{R' - 1}, \tilde K_0, \ldots, \tilde K_{d - 1} \}
$$ 
which further diagonalizes the $(2, 2)$ blocks. Let $\mathscr{K}$ be the sub-bundle generated by $\{ \tilde K_0, \ldots, \tilde K_{d - 1} \}$. The frame $\{ \tilde T_0, \ldots , \tilde T_{R' - 1} \}$ is also a frame of 
$$
\mathscr{T} = \mathscr{K}^{\perp},
$$ 
the orthogonal sub-bundle with respect to the polarized pairing (cf.~Lemma \ref{metrical}). By \S \ref{s:5.2}, $\mathscr{F}$ is defined in variables $u = x^{1/d}$ and $y, z$. For convenience we denote the corresponding cyclic extension of the Novikov ring $\mathscr{R}'$ by 
$$
\widetilde{\mathscr{R}'} := \mathscr{R}' [u]/(u^d - x).
$$
Our constructions above are  over the ring $\widetilde{\mathscr{R}'}[\![z]\!]$. 

Denoted by $\mathscr{T}_0$ and $\mathscr{K}_0$ the restriction of $\mathscr{T}$ and $\mathscr{K}$ at $z = 0$ respectively. As in \S \ref{ss:abs}, for an element $f \in \widetilde{\mathscr{R}'}[\![z]\!]$ we write $f(0) = f(u, y, z = 0)$. 

By \eqref{e:z=0}, a simple solution to issue (i) is to restrict to the $z = 0$ slice which we will take in this subsection. A more sophisticated and complete solution needs the machinery of BF/GMT which will be done in the next section.

Issue (ii) is more subtle: let $\{T_\mu\}_{\mu = 0}^{R - 1} = \{T_{\mathbf{e}}, k_i\}$ be the constant frame constructed in Lemma \ref{l:1}. For a class $T_m$ and a divisor $D$, we have
\begin{equation} \label{e:ind}
\begin{split}
D* T_m &= D.T_m + \sum_{\mu = 0}^{R - 1} \langle D, T_m, T^\mu \rangle_+ T_\mu,
\end{split}
\end{equation} 
where $+$ stands for the invariants with non-trivial curve classes. By Lemma \ref{l:K} and \ref{l:C2}, $\langle D, T_m, T^\mu \rangle_+ \ne 0$ only in the following two cases: 
\begin{itemize}
\item[(1)] If $\mu \le R' - 1$ then $\deg T_\mu < \deg T_m$ (cf.~Lemma \ref{l:weight}). In this case the invariant is a scalar multiple of $xy$ or $y$. 
\item[(2)] If $\mu \in [R', R - 1]$ then $T_\mu = T_{R'} = k_0 = (\xi - h)^{r' + 1}$ and $T_m = T_{R - 1} = k_{d - 1} = h^{d - 1}(\xi - h)^{r' + 1}$. The invariant is $(D.\ell) (-1)^{r' + 1}/x$.
\end{itemize}
With (1) and (2), \eqref{e:ind} becomes
\begin{equation} \label{e:c-q}
\begin{split}
D* T_m &= D.T_m \\
&+ \sum_{\mu = 0}^{R' - 1} \langle D, T_m, T^\mu \rangle_+ T_\mu + \delta_{m, R - 1} (-1)^{r' + 1}\frac{(D.\ell) }{x} k_0,
\end{split}
\end{equation} 
where the sum can be restricted to the range $\deg T_\mu < \deg T_m$. 

Equation \eqref{e:c-q} leads to a recursive formula for $(D.T_m)*$, hence the polynomial expression of $T_\mu*$ in $h*$ and $\xi*$ as in \eqref{e:q-div}. For example, we have

\begin{lemma} \label{l:sing}
For $j \ge 1$, in the polynomial expression of $(h^j k_{d - 1})*$ in $h*$ and $\xi *$, the terms with singular coefficient arise from  
$$
\frac{(-1)^{r' + 1}}{x} (h*)^{j - 1} * k_0 * = \frac{(-1)^{r' + 1}}{x} (h*)^{j - 1} * ((\xi - h)*)^{r' + 1} + O(y).
$$

For $k_i \in K$, no singular coefficients occur for $k_i*$.
\end{lemma}

\begin{lemma} \label{l:reg}
For any $\alpha \in H(X)$, the matrix for $(\xi.\alpha)*$ has no singular entries in $x, y$. Also the $(1, 1)$ block of the matrix $\tilde C_\mu$ for $T_\mu *$, $\mu \in [0, R' - 1]$, has no singular entries in $x, y$. This resolves issue (ii).
\end{lemma}

\begin{proof}
The first statement follows from \eqref{e:c-q} and induction since $(\xi.\ell) = 0$. For the second statement, notice that the constant frame $T_\mu$ in Lemma \ref{l:1} has the property that whenever there is a correction term by $k_{j - (r' + 1)}$ given in \eqref{e:basis}, then $T_\mu$ contains the factor $\xi$. The result follows.
\end{proof}

Now we may derive the splitting of small quantum rings: 

\begin{proposition} \label{p:split}
Let $\mathscr{K}_0$ be the sub-bundle generated by $\{ \tilde K_0(0), \ldots, \tilde K_{d - 1}(0) \}$. Then both $\mathscr{K}_0$ and $\mathscr{T}_0 = \mathscr{K}_0^\perp$ are ideals of $Q_0 H(X)\otimes \mathscr{R}'$ and
\begin{equation}
\begin{split}
Q_0 H(X)\otimes \mathscr{R}' \cong \mathscr{T}_0 \times \mathscr{K}_0 \cong_{\widetilde{\mathscr{R}'}} \langle \tilde T_0(0), \ldots, \tilde T_{R' - 1}(0) \rangle \times \Bbb C^{d},
\end{split}
\end{equation}
where $\mathbb{C}^{d}$ is the trivial ring consisting of $d$ idempotents. 

The second isomorphism is valid only over the extension $\widetilde{\mathscr{R}'}$ of $\mathscr{R}'$.
\end{proposition}


\begin{proof} 
There exist $C_{\mu \nu}^{\rho}$ and
distinct eigenvalues $\Lambda_{\mu i}$ and $\Lambda_i \neq 0$ such that
\begin{equation} \label{e:str}
\begin{split}
 \tilde T_{\mu}(0) * \tilde K_i(0) &= \Lambda_{\mu i} \tilde K_i(0), \qquad \forall \mu = 0, \ldots, R' - 1,\\
 \tilde K_i(0) * \tilde K_j(0) &=  \delta_{ij} \Lambda_i \tilde K_i(0) , \qquad \forall i = 0, \ldots, d - 1,\\
 \tilde T_{\mu}(0) * \tilde T_{\nu}(0) &=\sum_{\rho=1}^{R'}  C_{\mu \nu}^{\rho} \tilde T_{\rho}(0).
\end{split}
\end{equation}
In fact, $\Lambda_{\mu i} =0$ due to the self-duality of $\mathscr{K}$ and the Frobenius property
\[
 0= (\tilde T_{\mu}, \tilde K_i * \tilde K_j)(0) = (\tilde T_{\mu} * \tilde K_i, \tilde K_j)(0) = \Lambda_{\mu i} (\tilde K_i, \tilde K_j)(0)
\]
for all $i, j, \mu$. The second equality in \eqref{e:str} follows from 
$$
\Lambda_{ij} \tilde K_j(0) = \tilde K_i(0) * \tilde K_j(0) = \tilde K_j(0) * \tilde K_i(0) = \Lambda_{ji} \tilde K_i(0)
$$ 
and hence $\Lambda_{ij} = \delta_{ij} \Lambda_i$. It also follows that $\epsilon_i := \tilde K_i(0)/\Lambda_i$ is an idempotent for each $i$ since $\epsilon_i * \epsilon_j = \delta_{ij} \epsilon_j$. 

We need to show that $C_{\mu \nu}^\rho \in \tilde{\mathscr{R}}'$: the block diagonalization gives
$$
\tilde T_\mu(0) * = T_\mu * + \sum f_\mu^i(u, y) \tilde K_i(0) *
$$   
for some $f_\mu^i \in \tilde{\mathscr{R}}'$. By Lemma \ref{l:sing} and \eqref{e:EV} in Remark \ref{e:EV}, the matrix for the last term has entries in $\tilde{\mathscr{R}}'$. And by Lemma \ref{l:reg}, the same holds for $T_\mu *$.

We also need to show that $\epsilon_j \in \mathscr{K}_0$. By Lemma \ref{l:irr-root}, the eigenvalue function for $(\xi - h)*$ on $\tilde K_j(0)$, with $K_j = \sum_{l = 0}^{d - 1} \omega^{-jl} u^l k_l$ being given by \eqref{e:EV}, has its leading terms being
$$
\mu_j - \lambda_j = \omega^{-j(r + 1)}  u^{r + 1} y - \frac{\omega^j}{u} = - \frac{\omega^j}{u} (1 - \omega^{-j(r + 2)}  u^{r + 2} y).
$$ 
By \eqref{e:str} and \eqref{e:EV} again, the leading terms of $\Lambda_j(u, y)$ is then given by 
$$
\Big( \sum_{l = 0}^{d - 1} \omega^{-jl} u^l \frac{\omega^{jl}}{u^l} \Big) (-1)^{r' + 1} \frac{\omega^{j(r' + 1)}}{u^{r' + 1}} = d (-1)^{r' + 1} \frac{\omega^{j(r' + 1)}}{u^{r' + 1}}.
$$
Hence $1/\Lambda_j \in \tilde{\mathscr{R}}'$ and $\epsilon_j = \tilde K_j(0)/\Lambda_j$ is a regular vector field over $\tilde{\mathscr{R}}'$. This shows the splitting of quantum product at $z = 0$:
$$
Q_0 H(X) \otimes \tilde{\mathscr{R}}' \cong \langle \tilde T_0(0), \ldots, \tilde T_{R' - 1}(0) \rangle \times \Bbb C^{r - r'}.
$$

It remains to observe that while the frame $\{\tilde T_\mu, \tilde K_i \}$ is defined over $\widehat{\mathscr{R}}'$, the bundles $\mathscr{T}$ and $\mathscr{K}$ are actually defined over $\mathscr{R}'$. the proof is complete.
\end{proof}

We note that $\epsilon_j$ vanishes along the divisor $u = 0$. 


\section{Existence of $\hat \Psi$ as an $F$-embedding of $QH$} \label{s:small}

In the above the quantum product $*$ is performed in $H(X)$. To get the quantum product $*'$ on $H(X')$ we apply the BF/GMT procedure on the $(1, 1)$ blocks $\tilde C_\mu^{11}$'s. By Proposition \ref{p:5.6} below, this produces $QH(X')$ along certain locus $\sigma(\hat \bs)$ which is a non-linear map over the small parameter space $\hat \bs \in H^{\le 2}(X')$. In particular we get an isomorphism
\begin{equation} \label{e:qiso}
\langle \tilde T_0(0), \ldots, \tilde T_{R' - 1}(0) \rangle \cong \sigma^* QH(X')
\end{equation}
in a suitable sense---it is not a ring isomorphism since $\tilde T_0(0)$ is not the identity element in $\mathscr{T}$. Efforts will be paid to modify this isomorphism, or rather the frame $\tilde T_i$'s, to achieve a ``ring isomorphism'' (cf.~Theorem \ref{t:small}). 

It turns out that the correct category to state this isomorphism for the full (big) quantum cohomology is the category of $F$-manifolds. This is worked out in \S \ref{ss:F}. (See in particular Proposition \ref{p:int-mfd} and equation \eqref{e:Psi}.)

\subsection{Birkhoff factorization and generalized mirror transform}

\begin{proposition} \label{p:5.6}
After the Birkhoff factorization and generalized mirror transformation $\sigma(\hat \bs) \in H(X') \otimes \mathscr{R}'$ with $\hat \bs \in H^{\le 2}(X')$, $\tilde C_1^{11}$ and $\tilde C_2^{11}$ become the corresponding connection matrices for the quantum cohomology on $X'$ along $\sigma(\hat \bs)$.
\end{proposition}

\begin{proof}
The Picard--Fuchs (higher order) equations on $X$ have coefficients
as polynomials in $q_1$ and $q_2$, and similarly for $X'$. When restricting the variables to $P^1_{q_1} \setminus \{0, \infty \}$, by Lemma \ref{l:1.2}, these two systems are equivalent.

The Picard--Fuchs (first order) system on $X$ is entire, with irregular singularity of order 1
at $q_1 = \infty$.
What we have done to the Picard--Fuchs system of $X$ is to perform
gauge transformations and then block diagonalization to remove
the irregular part at $q_1 =\infty$.
The regular singular part still satisfies the same PF equations up to gauge transformation.

Since the flat connection of the quantum cohomology of $X'$ is equivalent to
the above PF system along the small parameter space $\hat \bs \in H^{\le 2}(X')$ up to Birkhoff factorization (gauge transformation) and generalized mirror transformation $\hat \bs \mapsto \sigma(\hat \bs) \in H(X')$ matching the initial conditions (due to Iritani and Coates--Givental \cite{CG}), the resulting system must be equivalent up to BF and GMT. 

After the BF (and GMT), the connection matrices are independent of $z$.
However, the \emph{frame} (in terms of constant vectors in $H(X)$) might still
have apparent $z$ dependence. The new frame is to be identified with the constant frames in $H(X')$, which establishes the desired correspondence. 
\end{proof}

Below we review the process of BF/GMT in the current situation aiming at a better understanding of Proposition \ref{p:5.6} (and the isomorphism \eqref{e:qiso}). \smallskip


Since the original $C_0 = {\rm Id}$ (corresponding to $T_0 *$) on $H(X)$, we have also $\tilde C_0^{11} = {\rm Id}$ on $H(X')$. So in practice it is sufficient to perform the BF/GMT only on $\tilde C_a^{11}$ for $a = 1, 2$. 

Let $B_1 = B_1(x, y, z)$ be the BF matrix and set $B_1(0) = B_1(x, y, 0)$. Consider the small parameter
$$
\hat \bs = s^0 T_0' + s^1 h' + s^2 \xi' \in H^0(X') \oplus H^2(X').
$$ 
From $\hat \bs = \T \hat \bt = t^0 + t^1(\xi' - h') + t^2 \xi'$ we have identifications 
$$
s^0 = t^0,\quad s^1 = -t^1,\quad s^2 = t^1 + t^2.
$$ 
Then under the substitution $x = q^{\ell'}e^{s^1}$, $y = q^{\gamma'} e^{s^2}$, the ``$z$-free'' matrix
\begin{equation} \label{e:z-free}
\begin{split}
C_a'(\hat \bs) &= -(z \p_a B_1) B_1^{-1} + B_1 \tilde C_a^{11} B_1^{-1} \\
&= B_1(0) \tilde C_{a; 0}^{11} B_1(0)^{-1}(x, y), \qquad a \in \{0, 1, 2\}
\end{split}
\end{equation}
is related to the matrices $A_\mu'(\sigma)$ for $T_\mu' *'$ at the generalized mirror point $\sigma = \sigma(\hat \bs) \in H(X')[\![x, y]\!]$ via 
\begin{equation} \label{gmt'}
C_a'(\hat \bs) = \sum_\mu A_\mu'(\sigma(\hat \bs)) \frac{\p\sigma^\mu}{\p s^a}(\hat \bs), \qquad a = 0, 1, 2.
\end{equation}
In terms of the connection one form $A' = \sum_\mu A'_\mu \, d\sigma^\mu$, this is simply $\sigma^* A'$.

To proceed, it is convenient to consider the weight zero variables 
\begin{equation} \label{e:wt0}
s: = zu, \qquad t := u^{r + 2} y.
\end{equation}
This is not to be confused with the above flat coordinates $s^i$ and $t^i$.

\begin{lemma}
For the $\tilde C_a^{22}$ diagonalized block, the ``Birkhoff factorization $B_2$'' can be found for each $1 \times 1$ block by elementary integrations. 

More precisely, there is a weight zero power series $\phi_j(zu, u^{r + 2}y)$ for each $j \in [0, d - 1]$ such that the frame $\bK_j := \phi_j \tilde K_j$ satisfies
\begin{equation} \label{e:BFK}
zu \frac{\p}{\p u} \widehat{\bK}_j I = \frac{a^\circ_j(t)}{u} \widehat{\bK}_j I, \qquad zy \frac{\p}{\p y} \widehat{\bK}_j I = \frac{b^\circ_j(t)}{u} \widehat{\bK}_j I, 
\end{equation}  
where $a^\circ_j$ and $b^\circ_j$ are analytic in $t$ with $a^\circ_j(0) = \omega^j$ and $b^\circ_j(t) = \omega^{-j(r + 1)} t + \ldots$. 

Indeed, $\lambda_j = a^\circ_j/u$ (resp.~$\mu_j = b^\circ_j/u$) is the eigenvalue function of $h*_{small}$ (resp.~$\xi *_{small}$) with eigenvector $\bK_j(u, y, z = 0)$.
\end{lemma}

\begin{proof}
Each $1 \times 1$ equation is irregular of the form 
\begin{equation} \label{e:1by1}
zu \frac{\p}{\p u} \widehat{\tilde K}_j I = \frac{1}{u} a_j(zu, u^{r + 2} y) \widehat{\tilde K}_j I, \qquad zy \frac{\p}{\p y} \widehat{\tilde K}_j I = \frac{1}{u} b_j(zu, u^{r + 2} y) \widehat{\tilde K}_j I.
\end{equation}
In the $(s, t)$ coordinates we write $a_j(s, t) = a_{j}(0, t) + s \alpha_j(s, t)$ and it is elementary to see that there is a series $\phi_j(s, t)$ such that the equation for $\bK_j := \phi_j \tilde K_j$ eliminates $s r_j(s, t)$. Indeed the equation becomes
$$
s \frac{\p}{\p s} \widehat{\bK}_j I = \frac{1}{s}  a_{j}(0, t) \widehat{\bK}_j I + \Big( s \frac{\p}{\p s} \phi_j + \alpha_j(s, t) \phi_j \Big) \widehat{\tilde K}_j I = 0
$$
and $\phi_j(s, t)$ is solved from the regular equation 
$$
s \frac{\p}{\p s} \phi_j + \alpha_j(s, t) \phi_j  = 0.
$$ 
The initial condition $\phi_j(0, t)$ is selected so that 
\begin{equation} \label{e:ini}
t \frac{\p}{\p t} \phi_j(0, t) + \beta_j(0, t) \phi_j(0, t) = 0,
\end{equation}
where $b_j(s, t) = b_j(0, t) + s \beta_j(s, t)$. 

The compatibility of the system then implies that equation \eqref{e:ini} holds without setting $s = 0$, which is what we want to prove. The last statement is a general statement about the small quantum product. 
\end{proof}

\begin{remark}
As in the proof of Proposition \ref{p:split}, we have $\bK_i(0) * \bK_j(0) = \delta_{ij} \mathbf{\Lambda}_j \bK_j(0)$ at $\hat \bs$ for $\mathbf{\Lambda}_j(u, y) = \phi_j \Lambda_j$. The idempotents are  
\begin{equation} \label{e:idem}
\epsilon_j(u, y) = \bK_j(0)/\mathbf{\Lambda}_j(u, y),
\end{equation} 
hence the additional information provided by $\bK_j$ lies in \eqref{e:BFK}.
\end{remark}

Denote $T_i = \Tp  T'_i$ as before. We combine the block diagonalization $P$ and Birkhoff factorization $B = B_1 \oplus B_2$ into a single gauge transformation 
\begin{equation} \label{e:gauge}
G = PB^{-1} = [\widetilde\bT_0, \ldots, \widetilde\bT_{R' - 1}, \bK_0, \ldots, \bK_{d - 1}]
\end{equation} 
with $\widetilde\bT_i$ (resp.~$\bK_j$) being the resulting frames on $\mathscr{T}$ (resp.~$\mathscr{K}$) such that 
$$
\widetilde\bT_i \cong T_i \mod NE(X').
$$ 
Let $\widetilde\bT^{i} = \sum_l {\bf g}^{il} \widetilde\bT_l$ be the dual frame with respect to the pairing in \eqref{e:pairing}:
$$
{\bf g}_{i\bar j} = \Lb \widetilde\bT_i, \widetilde\bT_j \Rb.
$$ 

Since the connection matrices $C'_a(\hat s)$'s in \eqref{e:z-free} are $z$-free, the $(i, j)$-th entry is precisely the GW invariant in the frame at $z = 0$:
\begin{equation} \label{d:z=0}
\bT_i := \widetilde\bT_i(0).
\end{equation}
Hence 
\begin{equation} \label{e:entry}
(C'_{a})_j^i (\hat \bs) = ( T_a * \bT_j, \bT^i )^X (\hat \bs) \equiv \La T_a, \bT_j, \bT^i \Ra^X (\hat \bs).
\end{equation} 

Since $H^2(X)$ also generates $H(X)$ via small quantum product, we thus have (by WDVV equations) a slightly stronger vanishing result:

\begin{lemma} \label{l:van}
For any $a \in H(X)$, $\La a, \bT^i, \bK_j\Ra^X (\Tp \hat \bs) = 0 = \La a, \bK^j, \bT_i\Ra^X (\Tp \hat \bs)$. 
\end{lemma}

In terms of their $(i, j)$-th entries, equation \eqref{e:z-free} becomes
\begin{equation} \label{e:gmt'}
\La T_a, \bT_j, \bT^i \Ra^X (\hat \bs) = \sum_\mu \frac{\p\sigma^\mu}{\p s^a}(\hat \bs) \La T'_\mu, T'_j, T'^i\Ra^{X'} (\sigma(\hat \bs)). 
\end{equation}
Since $(A_\mu')^{i}_{0} = \delta_{\mu}^{i}$, by comparing the first column we find 
\begin{equation} \label{e:dsigma}
(C_a')^{\mu}_{0}(\hat \bs) = \La T_a, \bT_0, \bT^\mu \Ra^X (\hat \bs) = \frac{\p \sigma^\mu}{\p s^a}(\hat \bs).
\end{equation}
Also we have $\sigma \equiv \hat \bs$ modulo $q^{\ell'}$, $q^{\gamma'}$, equation \eqref{e:dsigma} then determines $\sigma(\hat \bs)$. 

Notice that $\p \sigma^\mu/\p s^0 = \delta^{\mu}_{0}$, but $\sigma^0(\hat \bs)$ depends on $\hat \bs$ non-trivially and $\sigma(0) \ne 0$. (See Corollary \ref{c:GMT} for an explicit example on $\sigma(\hat \bs)$.)

We may also rewrite \eqref{e:gmt'} (or rather \eqref{e:z-free} and \eqref{gmt'}) into its intrinsic form in Dubrovin connections.

\begin{proposition} \label{p:bbgm}
Along $\hat \bs \in H^{\le 2}(X')$ we have a canonical isomorphism
\begin{equation} \label{e:bbgm}
(\mathscr{T}, \nabla^X \mid_{\mathscr{T}}) \cong (H(X') \otimes \mathscr{R}', \sigma^* \nabla^{X'})
\end{equation}
of connections, where $\sigma: H^{\le 2}(X') \to H(X') \otimes \mathscr{R}'$ is uniquely determined by \eqref{e:bbgm} and the constraint that $\sigma(\hat \bs) \equiv \hat \bs \mod{NE(X')}$.
\end{proposition}

\subsection{Special quantum invariance under the normalized frame}

\subsubsection{Compatibility on quantum products via WDVV}

In order to deduce consequences on quantum products from Proposition \ref{p:bbgm}, the following lemma is the starting point. 

\begin{lemma} \label{l:D-mod}
The isomorphism in \eqref{e:bbgm} is compatible with the small quantum $D$-module structures. Equivalently the quantum products of divisor classes $T_a := \Tp T'_a$ on $X$ and of classes $\sigma_* T'_a$ on $X'$ are compatible along the small parameter $\hat s \in H^{\le 2}(X')$.
\end{lemma}

\begin{proof}
We first notice that 
\begin{equation} \label{e:T0'}
\sigma_* T'_0 = \sum_\mu \p \sigma^\mu/\p s^0 \, T'_\mu = \sum_\mu \delta^\mu_0 T'_\mu = T'_0.
\end{equation}

Take two divisor classes $T'_a, T'_b$. Then from the WDVV equations,
\begin{equation} \label{e:wdvv}
\La T_a * T_b, \bT^i, \bT_j \Ra^X = \sum_\lambda \La T_a, \bT^i, \bT_\lambda \Ra^X \La T_b, \bT^\lambda, \bT_j \Ra^X.
\end{equation}
Along the small parameters $\Tp \hat \bs$, by Lemma \ref{l:van}, the sum is non-zero only in the non-kernel indices ($T_\lambda \not\in K$). By \eqref{e:gmt'}, the sum then becomes
$$
\sum_\lambda \Big( \La \sigma_* T'_a, T'^i, T'_\lambda \Ra^{X'} \La \sigma_* T'_b, T'^\lambda, T'_j \Ra^{X'}\Big) (\sigma(\hat \bs)).
$$
By the WDVV equations on the $X'$ side we then conclude
\begin{equation} \label{e:comp}
\La T_a * T_b, \bT^i, \bT_j\Ra^X (\Tp \hat \bs) = \La \sigma_* T'_a *' \sigma_* T'_b, T'^i, T'_j\Ra^{X'}(\sigma(\hat \bs)),
\end{equation}
where the tangent map $\sigma_*$ is performed at $\hat s$ and the quantum product on the right-hand-side is on $X'$ at $\sigma(\hat \bs)$. 

By induction on $r \in \Bbb N$, the equation \eqref{e:comp} holds for $T_{a_1} * \ldots * T_{a_r}$ and $\sigma_* T'_{a_1} * \ldots * \sigma_* T'_{a_r}$. The proof is complete.
\end{proof}

\subsubsection{Pseudo identity and the normalized frame}

Recall that $\bT_0 \equiv T_0 = {\rm id} \mod NE(X')$. The next step is to transform $\bT_0$ to the identity element (section) $e$ in $\mathscr{T}$ and normalized $\bT_i$'s accordingly. 

For $T_k \in K^\perp$, by Lemma \ref{l:reg} we may represent 
\begin{equation} \label{e:Pk}
T_k *_{\Tp \hat \bs} = P_k(h*_{\Tp \hat \bs}, \xi*_{\Tp \hat \bs})
\end{equation}
where $P_k$ is a polynomial with coefficient in $x, y$. 

\begin{definition} \label{d:Jac}
We define the $\Bbb C[\![x, y]\!]$-valued $R' \times R'$ matrix $(\mathscr{J}_k^\mu)$ by
\begin{equation} \label{e:Pkmu}
\sum_{\mu = 0}^{R' - 1}\mathscr{J}_k^\mu  T'_\mu := P_k(\sigma_* (\xi' - h')*', \sigma_*\xi' *') *' T'_0,
\end{equation}
where the quantum product $*'$ is taken at $\sigma(\hat \bs)$. Note that $\mathscr{J}_a^\mu = \p \sigma^\mu/\p s^a$ for $a \in \{0, 1, 2\}$ and $\mathscr{J}_k^\mu \cong \delta_k^\mu \mod NE(X')$ for all $k$. Hence $(\mathscr{J}_k^\mu)$ is invertible.
\end{definition}

Then by Lemma \ref{l:D-mod}, or rather equation \eqref{e:comp}, we have 
\begin{equation} \label{e:Pk}
\begin{split}
&\La T_k, \bT^i, \bT_j\Ra^X (\Tp \hat \bs) = \La P_k(h*, \xi *) * T_0, \bT^i, \bT_j\Ra^X (\Tp \hat \bs) \\
&= \La P_k(\sigma_* (\xi' - h')*', \sigma_*\xi' *') *' \sigma_* T'_0, T'^i, T'_j \Ra^{X'}(\sigma(\hat \bs)) \\
&= \sum_\mu \mathscr{J}_k^\mu \La T'_\mu, T'^i, T'_j \Ra^{X'}(\sigma(\hat \bs)).
\end{split}
\end{equation}

\begin{lemma} \label{l:q-inv}
There is a unique element $\bS_0 \in \mathscr{T}$ such that $\bS_0 * \bT_0$ is the identity element (section) $e$ in $\mathscr{T}$ (and so $e$ acts as zero on $\mathscr{K}$). 
\end{lemma}

\begin{proof}
By our constructions, the structure constants $c_{kj}^i(u, y, z)$ defined by
$$
\bT_k * \bT_j = \sum c_{kj}^i \bT_i
$$
are series in $u, y, z$. In particular, by writing $\bS_0 = \sum_i w^j \,\bT_j$, then $\bS_0$ can be solved explicitly using the relation $\bT_0 * \bT_j = \sum_i c_{0j}^i \bT_i$. 
Indeed, from \eqref{e:idem}, the identity $e$ in $\mathscr{T}$ is given by
$$
e = T_0 - \sum_{j = 0}^{d - 1} e_j = \sum_i\varphi^i \bT_i
$$
for some series $\varphi^i(u, y, z)$ in $u, y, z$. So we need to solve the $R' \times R'$ linear system of equations
$$
\sum_{j = 0}^{R' - 1} c_{0j}^i\, w^j = \varphi^i, \qquad i = 0, 1, \ldots, R' - 1.
$$
Notice that, by Lemma \ref{l:reg} and the property that $B_1 \equiv {\rm Id} \pmod{x, y}$,
$$
c_{0j}^i = \La \bT_0, \bT^i, \bT_j \Ra^X = \La T_0, \bT^i, \bT_j \Ra^X + \sum_k f_0^k(u, y) \La T_k, \bT^i, \bT_j \Ra^X
$$ 
is a series in $u, y$ with $f_0^k(0, 0) = 0$. Also $\La T_0, \bT^i, \bT_j \Ra^X = \Lb \bT_j, \bT^i \Rb = \delta^i_j$. This shows that $(c_{0j}^i)$ is invertible and the lemma is proved.
\end{proof}

\begin{definition} \label{d:modf}
We call $\bS_0$ the pseudo-inverse of $\bT_0$, which is the inverse of $\bT_0$ in $\mathscr{T}$, and we define the \emph{normalized frame} 
$$
\breve \bT_\mu := \bT_\mu * \bS_0
$$ 
on $\mathscr{T}$. 
\end{definition}

Along $\Tp \hat s$, by setting $j = 0$ in \eqref{e:Pk} we find
\begin{equation} \label{e:bT0}
\bT_0 * T_k = \sum_\mu \La \bT_0, T_k, \bT^\mu \Ra^X \bT_\mu = \sum_\mu \mathscr{J}_k^\mu \bT_\mu.
\end{equation}

Applying $\bS_0 *$ to \eqref{e:bT0}, Lemma \ref{l:q-inv} then leads to the important

\begin{proposition}[Basic transformation rule]
For $k \in [0, R' - 1]$, we have
\begin{equation} \label{e:FC}
T_k = \sum_{\mu = 0}^{R' - 1} \mathscr{J}_k^\mu \breve \bT_\mu \pmod{\mathscr{K}}.
\end{equation}
In particular, the normalized frame $\breve \bT_\mu$ is defined over $x, y$.
\end{proposition}

\subsubsection{Special quantum invariance}

With \eqref{e:FC}, then equation \eqref{e:Pk} becomes 
\begin{equation} \label{e:eq0}
\La \bT_\mu, \bT^i, \bT_j\Ra^X (\Tp \hat \bs) = \La T'_\mu, T'^i, T'_j \Ra^{X'} (\sigma(\hat \bs)). 
\end{equation}

\begin{lemma}
With respect to the pairing $\breve{\bf g}_{i j} = (\breve\bT_i, \breve\bT_j )$, the dual frame $\breve \bT^i := \sum_j \breve{\bf g}^{i j} \breve\bT_j$ is given by $\breve \bT^i = \bT^i * \bT_0$.
\end{lemma}

\begin{proof}
Indeed, 
$$
(\bT_j * \bS_0,\bT^i * \bT_0 ) = (\bT_j * \bS_0 * \bT_0, \bT^i ) = (\bT_j, \bT^i ) = \delta^i_j.
$$
Here, the Frobenius property on the pairing is used. 
\end{proof}

Hence for any class $a$ we have
\begin{equation} \label{e:ch}
\begin{split}
\La a, \breve \bT_j , \breve \bT^i \Ra &= ( a * \breve \bT_j, \breve\bT^i ) = ( a * \bT_j * \bS_0 , \bT^i * \bT_0 ) \\
&= ( a * \bT_j * \bS_0 * \bT_0, \bT^i ) = ( a * \bT_j, \bT^i ) \\
&= \La a, \bT_j, \bT^i \Ra.
\end{split}
\end{equation}
Together with \eqref{e:eq0}, we arrive at a simple statement:

\begin{theorem} \label{t:small}
Under the $\Bbb C[\![NE(X')]\!]$-linear map 
$$
\sum a^i \breve \bT_i \mapsto \sum a^i T'_i,
$$ 
the quantum product on $\mathscr{T}$ at $\Tp \hat \bs \in H^{\le 2}(X)$ is isomorphic to the quantum product on $H(X')$ at $\sigma(\hat \bs) \in H(X') \otimes \Bbb C[\![NE(X')]\!]$. Namely
\begin{equation} \label{e:dframe}
\La \breve\bT_\mu, \breve\bT^i, \breve\bT_j\Ra^X (\Tp\hat \bs) = \La T'_\mu, T'^i, T'_j \Ra^{X'} (\sigma(\hat \bs))
\end{equation}
for all $0 \le i, j, \mu \le R' - 1$. 
\end{theorem}

The ``subring'', or rather ``ideal'', $(\mathscr{T}, *)$ of $Q_0 H(X)$ is not isomorphic to $Q_0 H(X')$ since $\sigma(0) \ne 0$ (cf.~Corollary \ref{c:GMT} for contributions from the extremal ray). Nevertheless a standard induction on Mori cone implies that

\begin{corollary}
The big quantum cohomology $QH(X')$ can be effectively computed from $QH(X)$ through equation \eqref{e:dframe}.
\end{corollary}

\subsection{Non-linear reconstructions}

In this subsection we will complete the proof of Theorem \ref{t:main} by constructing the embedding $\widehat \Psi$ with the imposed properties.

\subsubsection{Remarks on reconstructions over the big parameter spaces}

The complication in dealing with the GMT $\hat \bs \mapsto \sigma(\hat \bs)$ lies on the fact that it is a graph over the small parameters instead of an invertible transformation. The basic idea to resolve the problem is to apply suitable reconstruction theorems on $X$ and $X'$ respectively and to study the compatibility between them.

When the total cohomology $H$ is generated by $H^2$ under cup product, the reconstruction from 3-point genus zero GW invariants to all $n$-point genus zero invariants follows from the WDVV equations as done by Kontsevich--Manin. Under the same condition, a version in the setup of abstract quantum $\mathscr{D}$-modules was formulated and carried out by Iritani in \cite[Theorem 4.9]{Iritani}, which says that the ``abstract big QDM'' is naturally determined by the ``abstract small QDM''. The abstract version is suitable in our current context since it does not require inductions on the Mori cone. 

To trace the reconstruction procedure in all directions of $H'$ consistently, we set $\hat \bs = 0$ and keep only the Novikov variables $\{q'^1, q'^2\} = \{q^{\ell'}, q^{\gamma'}\}$ in equation \eqref{e:dframe} as the starting point. 
Namely we \emph{decouple} the roles played by $\hat \bs$ and $q'^a$'s back to the status they are in the definition in \eqref{e:q-prod}. 

Denote the resulting frames by $\Bbb T_i(q')$ and let $\sigma_0 = \sigma_0(q')$ be the generalized mirror point at $\hat \bs = 0$. Equation \eqref{e:gmt'}, via \eqref{e:ch}, takes the form
\begin{equation} \label{e:ori}
\langle T_a, \Bbb T_j, \Bbb T^i \rangle^X = \sum_\mu \Big( \delta^\mu_a + q'^a \frac{\p\sigma_0^\mu}{\p q'^a} \Big) \La T'_\mu, T'_j, T'^i\Ra^{X'} (\sigma_0).
\end{equation}
Here $\delta^\mu_a$ is inserted since $\p \sigma^\mu/\p s^a \equiv \delta^\mu_a \mod NE(X')$. Also \eqref{e:dframe} becomes
\begin{equation} \label{e:start}
\langle \Bbb T_\mu, \Bbb T_j, \Bbb T^i\rangle^X = \La T'_\mu, T'_j, T'^i \Ra^{X'} (\sigma_0).
\end{equation}
We regard \eqref{e:ori} as the connection matrix $A_a(q', \bs)$ at $\bs = 0$ for
$$
z \nabla_a = z \p_a - A_a \equiv z q'^a \frac{\p}{\p q'^a} - A_a, 
$$ 
and \eqref{e:start} as the connection matrix $\Omega_\mu(q', \bs)$ at $\bs = 0$ for 
$$
z\nabla_\mu = z\p_\mu - \Omega_\mu \equiv z\frac{\p}{\p s^\mu} - \Omega_\mu.
$$
Notice that the coordinates $s^0, \ldots, s^{R' - 1}$ are centered at $\sigma_0$ when viewing on the $H(X')$ side and centered at $0$ on the $H(X)$ side. Also while the matrices $A_a$ and $\Omega_\mu$ are \emph{identically the same} for $X$ and $X'$, their meaning in quantum product are taken in completely different manners.  

The flatness of $\nabla$ is equivalent to the WDVV equations
\begin{equation}
\begin{split}
& [A_a, A_b] = [A_a, \Omega_\mu] = [\Omega_\mu, \Omega_\nu] = 0, \\
& \p_a A_b = \p_b A_a, \quad \p_a \Omega_\mu = \p_\mu A_a, \quad \p_\mu \Omega_\nu = \p_\nu \Omega_\mu.
\end{split}
\end{equation}
Consider the ideal $\frak m = (s^0, s^1, \ldots, s^{R' - 1})$. By induction on $k \in \Bbb N$, we may 
\begin{itemize}
\item[(i)] solve $A_a(q', \bs) \pmod{\frak m^k}$ from $\p_\mu A_a = \p_a \Omega_\mu \pmod{\frak m^{k - 1}}$, and then 
\item[(ii)] solve $\Omega_\mu (q', \bs) \pmod{\frak m^k}$ as a polynomial in $\Omega_j (q', \bs)$'s $\pmod{\frak m^{k - 1}}$ and $A_a$'s $\pmod{\frak m^k}$. 
\end{itemize}

The starting case $k = 1$ for (ii) is essentially Theorem \ref{t:small}. The relevant formulas are \eqref{e:Pkmu} and \eqref{e:Pk} used in the proof of Lemma \ref{l:q-inv}. Indeed, let 
$$
\mathscr{I} (q') = (\mathscr{I}^k_\mu) := \mathscr{J}^{-1}
$$ 
be the inverse matrix of $(\mathscr{J}_k^\mu)$ which depends only on $q'$'s. Then at $\sigma_0$,
\begin{equation} \label{e:poly}
\begin{split}
T'_\mu *' &= \sum_k \mathscr{I}^k_\mu \, P_k(\sigma_* (\xi' - h')*', \sigma_*\xi' *') \\
&= \sum_k \mathscr{I}^k_\mu \, P_k(A_2 - A_1, A_2),
\end{split}
\end{equation}
and $\Omega_\mu(q', 0)$ is given by \eqref{e:dframe} via \eqref{e:FC}.

Thus it remains to understand the geometric meanings on both sides under the WDVV reconstruction. On $X'$ this is standard and it leads to 
$$
(\Omega_\mu)_j^i (q', \bs) = \La T'_\mu, T'_j, T'^i \Ra^{X'} (\sigma_0 + \bs).
$$ 
In particular $(\Omega_\mu)_0^i (q', \bs) = \delta_\mu^i$ since $T'_0$ is the identity. 

On $X$ the reconstruction is \emph{not linear}---in each step of (ii) the identity section $\Bbb T_0(q', \bs) \pmod{\frak m^k}$ receives new correction terms. With this modification been done for each $k$, which is hard, the resulting structure should then lead to deformations of the embedding $\Tp :H(X') \hookrightarrow H(X)$ to certain $\widehat \Tp(q', \bs)$ which relates quantum products of $X$ and $X'$. 

%
%
%

When the GW theory under consideration is analytic, alternatively we may view WDVV as a Frobenius integrability condition in the context of integrable distributions and to construct $\widehat\Tp$ through certain ``canonical coordinates''. We will take this approach in the next section, and it is best described in terms of the notion of $F$-manifolds. 

\subsubsection{Integrable distribution and the canonical coordinates} \label{ss:F}

Recall that an $F$-manifold $M$ is a complex manifold equipped with a commutative and associative product structure on each tangent space $T_p M$, such that a WDVV-type integrability condition is forced when $p$ varies. In the context of quantum cohomology, this is simply the structure which remembers the quantum product but forgets the metric $g_{ij}$, and with a coordinate-free form of the WDVV (integrability) equations. 

Indeed, viewing the quantum product $*$ as a $(2, 1)$ tensor, Hertling and Manin (cf.~\cite[Definition 2.8, Theorem 2.14, 2.15]{Hertling} had shown that the WDVV equations can be rewritten as 
\begin{equation} \label{e:HM}
L_{X * Y} * = X * L_Y * + Y * L_X *
\end{equation}
for any local vector fields $X$ and $Y$, where $L$ denotes the Lie derivatives. In explicit terms this means that for any local vector fields $X, Y, Z, W$ we have
\begin{equation} \label{e:HMe}
\begin{split}
&[X*Y, Z*W] - [X*Y, Z]*W - [X*Y, W]*Z \\
&\quad = X*[Y, Z*W] - X*[Y, Z]*W - X*[Y, W]*Z \\
&\quad\qquad + Y*[X, Z*W] - Y*[X, Z]*W - Y*[X, W]*Z. 
\end{split}
\end{equation}

To apply it to our flip situation, we denote by $\mathcal{K}$ the irregular eigenbundle and its orthogonal complement $\mathcal{T} = \mathcal{K}^\perp$ the regular eigenbundle which extend the corresponding $\mathscr{K}$ and $\mathscr{T}$ from $\bs = 0$ to the big parameter space.

\begin{lemma} \label{l:cont}
Both $\mathcal{K}$ and $\mathcal{T}$ and the irregular/regular decomposition of the big quantum product on $T H_{\mathscr{R}'}$ are defined over the big parameter space $H_{\mathscr{R}'}$ over a punctured neighborhood of $q^{\ell'} = 0$.
\end{lemma}

\begin{proposition} \label{p:int-mfd}
The regular eigenbundle $\mathcal{T}$ is an integrable distribution of the relative tangent bundle $T H_{\mathscr{R}'}$. 

In particular, the image of $\widehat{\Tp}$ is the integrable submanifold $\M$ (over $\mathscr{R}'$) containing the slice $(q^{\ell'} \ne 0, \bt = 0)$ which contains the classical correspondence when modulo $\mathscr{R}'$.
\end{proposition}

\begin{proof}
Let $X, Z$ be any two local vector fields valued in $\mathcal{T} = \mathcal{K}^\perp$. Let $Y = e_i$ and $W = e_j$ be two idempotents valued in $\mathcal{K}$. Since $a * b = 0$ for any $a \in \mathcal{K}$ and $b \in \mathcal{K}^\perp$, \eqref{e:HMe} becomes
\begin{equation}
0 = - X * Z * [e_i, e_j] - \delta_{ij} e_j * [X, Z]. 
\end{equation}
Let $i = j$ we conclude that $e_j * [X, Z] = 0$ for all $j$. Hence $[X, Z] \in \mathcal{K}^\perp$.
\end{proof}

\begin{remark}
The above proof requires only that $\mathscr{K}$ contains no nilpotent sections, i.e.~generically semi-simple. Hence Proposition \ref{p:int-mfd} works in the global case as well, though in the formal setting. In the local case all the local models are toric and the analyticity is known (by Iritani), thus the Frobenius theorem needed is the classical one. In the global case we need to invoke the Frobenius theorem in the formal setting.  
\end{remark}

Now we use the full strength of the local model structure. The quantum product on the Frobenius manifold $H(X') \otimes \mathscr{R'}$ is semi-simple. Deonte by the idempotent vector fields on $H(X') \otimes \mathscr{R}'$ by $v'_0, \ldots, v'_{R' - 1}$. A well-known result of Dubrovin \cite[Main Lemma (3.47)]{Du} says that canonical coordinates exist. In our setting, we apply it in a family in $q'$ with center at $\sigma_0(q')$:

\begin{lemma}  
We have $[v'_i, v'_j] = 0$ for all $0 \le i, j \le R' - 1$. Hence the corresponding \emph{canonical coordinates} $u'^0, \ldots, u'^{R' - 1}$ satisfying 
$$
(u'^i(q', \bs = 0)) = \sigma_0(q')
$$ 
and $v'_i = \p/\p u'^i$ exist. 
\end{lemma}

Dubrovin's result was extended to $F$-manifolds by Hertling \cite[Theorem 2.11]{Hertling}. In our setting, the $F$-manifold $\M$ is semi-simple (or massive) in the sense that the quantum product on $T_p \M$ for $p \in \M$ is semi-simple. Denote the idempotent vector field be $v_1. \ldots, v_{R'}$. 

\begin{lemma} 
We have $[v_i, v_j] = 0$ for all $0 \le i, j \le R' - 1$. Hence the canonical coordinates $u^0, \ldots, u^{R' - 1}$ on $\M$ exist in the sense that $v_i = \p/\p u^i$. 
\end{lemma} 

We emphasize that we have constructed an analytic family of coordinate systems $(u^0(q', p), \ldots, u^{R' - 1}(q', p))$ parametrized by $q' \in \mathscr{R}'$. Write 
\begin{equation} \label{e:lc'}
\Bbb T_i(q') = \sum_{j = 0}^{R' - 1} a_i^j(q') \,v_j(q', \bs = 0)
\end{equation}
for an invertible $R' \times R'$ matrix $(a_i^j(q'))$. From Theorem \ref{t:small} (or \eqref{e:start}), we see easily that the same linear combination passes to the $X'$ side:
\begin{lemma}
After a possible reordering, we have
\begin{equation} \label{e:lc}
T'_i = \sum_{j = 0}^{R' - 1} a_i^j(q') \,v'_j(\sigma_0(q')),
\end{equation}
for all $i = 0, \ldots, R' - 1$.
\end{lemma}

Now we may define the map $\hat \Psi$ by matching the canonical coordinates. Namely, $\hat\Psi(q', \bs) \in \M$ is the unique point on $\M$ so that 
\begin{equation} \label{e:Psi}
u^i(\hat \Psi(q', \bs)) = u'^i(q', \bs) = u'^i(\sigma_0(q') + \bs)
\end{equation}
for $i = 0, \ldots, R' - 1$. Since the tangent map $\hat \Psi_*$ matches the idempotents
$$
\hat \Psi_* \frac{\p}{\p u'^i} = \frac{\p}{\p u^i},
$$
it induces a product structure isomorphism, and hence an $F$-structure isomorphism by \eqref{e:HMe}. Also along $s = 0$, by \eqref{e:lc'} and \eqref{e:lc} we have
$$
\hat\Psi_* T'_i = \Bbb T_i
$$
which matches the initial condition along the $\mathscr{R}'$-axis. 

At the beginning $\hat \Psi$ exists only locally. But since $H(X')$ is contractible, it exists globally by gluing the local maps. This completes the proof of Theorem \ref{t:main}.

%

\section{Exact determination of the Dubrovin connection}
The main purpose in this section is to observe the extremely nice phenomenon that we are able to modify the basis given by the ``quantized version'' of the basis given in Lemma~\ref{l:1} ``in a canonical manner'' to get the Gromov--Witten system on $X$ directly, without going through the BF/GMT process!

\subsection{Dubrovin connection in the $\Psi$-corrected quantum frame} \label{ss:explicit}

Now we are going to rewrite the higher order PDEs (PF on $X$) in terms of systems of first order PDE's
\begin{equation} \label{e:2.4}
\begin{split}
 (\zp_i - C_i (q_1, q_2))S = 0, \qquad i = 1,2,
 \end{split}
\end{equation}
such that $C_i$'s are \emph{polynomials in $q_1, q_2$} and are \emph{independent of $z$}. Here we think of $S$ as the $R \times R$ fundamental solution matrix. Also we keep the notation $C_i$'s though we eventually will show that they are precisely $A_i$'s.  

\begin{remark}
Before we perform the calculations, it is important to point out that in reducing the order of differentiations via the PF system, variables $q_1$, $q_2$ are created in the middle of a formula. It is clear that $\zp_i\, q_j = q_j\, \zp_i$ if $i \ne j$. However, for a term like $(\zp_i)^j q_i$, after commutation we get
$$
(\zp_i)^j q_i = q_i (\zp_i + z)^j.
$$  
In this way non-trivial $z$-dependence occurs, which is not allowed in the matrices $C_i$'s. The trick to avoid such a situation is to perform commutations only for terms of the form
$$
(\zp_1 - \zp_2)^j q_1 q_2 = q_1 q_2 (\zp_1 - \zp_2)^j.
$$
We will see that this is always possible for our choice of quantum basis.
\end{remark}

\subsubsection{The matrix for $z\p_1$} \label{ss:zp1} We will complete the proof of Lemma \ref{l:K}.

For $e_1 + e_2 \in [0, r' - 1]$, 
\begin{equation*}
\begin{split}
\zp_1v_{(e_1,\, e_2)} &= v_{(e_1+1,\, e_2)}. 
\end{split}
\end{equation*}
And for the boundary case $e_1 + e_2 = r'$,
\begin{equation} \label{e:edge}
\begin{split}
\zp_1v_{(e_1,\, e_2)} &= v_{(e_1 + 1,\, e_2)} -  (-1)^{r' - e_2} K_1.
\end{split}
\end{equation}
We emphasize here that $r > r'$ is essential for \eqref{e:edge} to be valid.

Also, if either $e_1 + e_2 \in [r' + 1, r]$, or $e_1 + e_2 \in [r + 1, r + r']$ and $e_1\leq r$, then we have
\begin{equation*}
\begin{split}
\zp_1v_{(e_1,\, e_2)} &= v_{(e_1+1,\, e_2)} .
\end{split}
\end{equation*}

For the remaining part, we find
\begin{equation*}
\begin{split}
& \quad (\zp_1)^{r-r'+i+j+1}\zp_2(\zp_2-\zp_1)^{r'-j}I \\
&= (\zp_1)^{r-r'+i+j+1}(\zp_2-\zp_1)^{r'-j+1} I+ (\zp_1)^{r-r'+i+j +2}(\zp_2-\zp_1)^{r'-j}I\\
&=  v_{(r-r'+i+j+1,\, r'-j+1)} + v_{(r-r'+i+j+2,\, r'-j)},\\
\end{split}
\end{equation*}
where $1\leq i, j\leq r'-1$ and $i+j \leq r'-1$.  Also,
\begin{equation*}
\begin{split}
& \quad (\zp_1)^{r-r'+i+1}\zp_2(\zp_2-\zp_1)^{r'}I \\
&= (\zp_1)^{r-r'+i+1}(\zp_2-\zp_1)^{r'+1} I + (\zp_1)^{r-r'+i+2}(\zp_2-\zp_1)^{r'}I \\
&= v_{(r-r'+i+2,\, r')} \\
\end{split}
\end{equation*}
and
\begin{equation*}
\begin{split} 
(\zp_1)^{r+1}\zp_2I & =\zp_2\, q_1(\zp_2-\zp_1)^{r'+1}I \\
&= q_1q_2 I,\\
(\zp_1)^{r+1}\zp_2(\zp_2-\zp_1)^iI & = (\zp_2-\zp_1)^i\zp_2\, q_1(\zp_2-\zp_1)^{r'+1}I \\
&= (\zp_2-\zp_1)^i q_1q_2 I \\
&= q_1q_2(\zp_2-\zp_1)^i I \\
&=q_1q_2\, v_{(0,\,i)}.
\end{split}
\end{equation*}
 Hence we get \small
\begin{equation*}
\begin{split}
&\quad \zp_1v_{(r+1,\, 0)}\\
&= (\zp_1)^{r+2} I + (-1)^{r'}(\zp_1)^{r-r'+1}(\zp_2-\zp_1)^{r'+1}I\\
&= (\zp_1)^{r-r'+1}((\zp_1)^{r'+1} + (-1)^{r'}(\zp_2 - \zp_1)^{r'+1}) I\\
&= (-1)^{r'} (\zp_1)^{r-r'+1} \zp_2((\zp_2-\zp_1)^{r'} - \zp_1(\zp_2-\zp_1)^{r'-1} + \cdots + (-1)^{r'} (\zp_1)^{r'})I\\
&= (-1)^{r'}(v_{(r-r'+2,\, r')} - (v_{(r-r'+2,\, r')} + v_{(r-r'+3,\, r'-1)}) + \cdots + (-1)^{r'} q_1q_2\, I)\\
&= (-1)^{r'}((-1)^{r'-1}v_{(r+1,\, 1)} + (-1)^{r'} q_1q_2\, I)\\
&= q_1q_2I - v_{(r+1,\, 1)},
\end{split}
\end{equation*} \normalsize
where the elementary formula $a^n - b^n = (a - b) (a^{n - 1} + a^{n - 2} b + \cdots + b^{n - 1})$ is used in deriving the third equality.

Similarly, for $1\leq i\leq r'-1$, \small
\begin{equation*}
\begin{split}
& \quad \zp_1 v_{(r+1,\, i)}\\
&= (\zp_1)^{r+2}(\zp_2-\zp_1)^i I + (-1)^{r'-i}(\zp_1)^{r-r'+1+i}(\zp_2-\zp_1)^{r'+1}I\\
&= (\zp_1)^{r-r'+1+i}(\zp_2-\zp_1)^i((\zp_1)^{r'+1-i} + (-1)^{r'-i}(\zp_2 - \zp_1)^{r'+1-i})I\\
&= (-1)^{r'}(\zp_1)^{r-r'+1+i}\zp_2(\zp_2-\zp_1)^i((\zp_2-\zp_1)^{r'-i} - \cdots + (-1)^{r'-i}(\zp_1)^{r'-i})I\\
&= (-1)^{r'}(v_{(r-r'+i+2,\, r')} - (v_{(r-r'+i+2,\, r')} + v_{(r-r'+i+3,\, r'-1)}) + \cdots + (-1)^{r'-i}q_1q_2\, v_{(0,\, i)})\\
&= (-1)^{r'}((-1)^{r'-i-1}v_{(r+i+1,\, 1)} + (-1)^{r'-i} q_1q_2\, v_{(0,\, i)})\\
&= (-1)^i(q_1q_2\, v_{(0,\, i)}-v_{(r+i+1,\, 1)}).
\end{split}
\end{equation*} \normalsize
For the last one with $i = r'$, we have
\begin{equation*}
\begin{split}
&\quad \zp_1v_{(r+1,\, r')}\\
&= (\zp_1)^{r+2}(\zp_2-\zp_1)^{r'}I + (\zp_1)^{r+1}(\zp_2-\zp_1)^{r'+1}I\\
&= (\zp_1)^{r+1}(\zp_2-\zp_1)^{r'}(\zp_1+\zp_2-\zp_1)I\\
&= q_1q_2(\zp_2-\zp_1)^{r'}I \\
&=  q_1q_2\, v_{(0,\, r')}.
\end{split}
\end{equation*}

Together with the calculations on the $K$ part in \eqref{q1-appear}, the proof of Lemma \ref{l:K} is thus complete. 

\subsubsection{The matrix for $z\p_2$} \label{ss:zp2}{\ } \smallskip

It suffices to determine the matrix $C_2(q_1,q_2) - C_1(q_1,q_2)$. 

By definition, for $e_1+e_2\leq r'-1$, 
\begin{equation*}
\begin{split}
(\zp_2 - \zp_1) v_{(e_1,\, e_2)} &= v_{(e_1,\, e_2 + 1)}, 
\end{split}
\end{equation*}
and for $e_1 + e_2 = r'$,
\begin{equation} \label{e:211}
\begin{split}
(\zp_2 - \zp_1) v_{(e_1,\, e_2)} &= v_{(e_1,\, e_2 + 1)} - (-1)^{r' - e_2 - 1} K_1.
\end{split}
\end{equation}

Note that 
\begin{equation*} 
\begin{split}
(\zp_2-\zp_1)^{r'+2}I
&= \zp_2(\zp_2-\zp_1)^{r'+1} I- \zp_1(\zp_2-\zp_1)^{r'+1}I \\
&= q_2\,I - K_2.
\end{split}
\end{equation*}
Then for $e_1 + e_2 \in [r' + 1, 2r' + 1]$ and $e_2\leq r' - 1$, we have
\begin{equation*}
\begin{split}
&\quad (\zp_2-\zp_1)v_{(e_1,\, e_2)} \\
&= (\zp_1)^{e_1} (\zp_2 - \zp_1)^{e_2+1} I + (-1)^{r'-e_2} (\zp_1)^{e_1 + e_2 - r'-1} (\zp_2 - \zp_1)^{r'+2} I\\
&= (\zp_1)^{e_1} (\zp_2 - \zp_1)^{e_2+1} I +(-1)^{r'-e_2-1} (\zp_1)^{e_1 + e_2 - r'} (\zp_2 - \zp_1)^{r'+1} I \\
&\qquad \qquad +(-1)^{r'-e_2} q_2(\zp_1)^{e_1 + e_2 - r'-1}I\\
&= v_{(e_1,\, e_2+1)} +(-1)^{r'-e_2} q_2\, v_{(e_1+e_2-r'-1,\, 0)}.
\end{split}
\end{equation*}
And, for $e_1 + e_2 \in [2r' + 2, r' + r]$ and $e_2 \leq r' - 1$ we have
\begin{equation*}
\begin{split}
&\quad (\zp_2-\zp_1)v_{(e_1,\, e_2)} \\
&= (\zp_1)^{e_1} (\zp_2 - \zp_1)^{e_2+1} I + (-1)^{r'-e_2} (\zp_1)^{e_1 + e_2 - r'-1} (\zp_2 - \zp_1)^{r'+2} I\\
&= v_{(e_1,\, e_2+1)} + (-1)^{r'-e_2} q_2(v_{(e_1+e_2-r'-1,\,0)} - (-1)^{r'}K_{e_1+e_2-2r'-1})\\
&= v_{(e_1,\, e_2+1)} + (-1)^{r'-e_2} q_2\, v_{(e_1+e_2-r'-1,\, 0)} + (-1)^{e_2-1}q_2\, K_{e_1+e_2-2r'-1}.
\end{split}
\end{equation*}

For the remaining part, if $e_1 \in [1, r' + 1]$ then
\begin{equation*}
\begin{split}
&\quad (\zp_2-\zp_1)v_{(e_1,\, r')}\\
&= (\zp_1)^{e_1} (\zp_2 - \zp_1)^{r'+1} I + (\zp_1)^{e_1 -1} (\zp_2 - \zp_1)^{r'+2} I \\
&= q_2(\zp_1)^{e_1-1}I \\
&= q_2\, v_{(e_1-1,\, 0)}.
\end{split}
\end{equation*}
And if $e_1 \in [r' + 2, r + 1]$ then 
\begin{equation*}
\begin{split}
(\zp_2-\zp_1)v_{(e_1,r')}
&= q_2 (\zp_1)^{e_1-1}I \\
&= q_2\, v_{(e_1-1,\, 0)} - (-1)^{r'} q_2\, K_{e_1-1-r'}.
\end{split}
\end{equation*}

Finally, for $i \in [0, r - r' - 2]$,
\begin{equation*} \label{q2-app}
\begin{split}
(\zp_2-\zp_1)K_{i+1} &= (\zp_1)^{i}(\zp_2-\zp_1)^{r'+2}I \\
&= q_2 (\zp_1)^i I - (\zp_1)^{i+1}(\zp_2-\zp_1)^{r'+1}I\\
&= q_2\, v_{(i,0)} - \delta_{(i,0)}(-1)^{r'}q_2\, K_{i-r'} - K_{i+2},
\end{split}
\end{equation*}
and 
\begin{equation} \label{q1-app2}
\begin{split}
(\zp_2-\zp_1)K_{r-r'} &= q_2 (\zp_1)^{r-r'-1} I - (\zp_1)^{r-r'}(\zp_2-\zp_1)^{r'+1}I\\
&= q_2\, v_{(r-r'-1,\, 0)} - \delta_{(r-r'-1,\, 0)}(-1)^{r'} q_2\, K_{r-2r'-1} \\
&\qquad \qquad + (-1)^{r'}q_1\, K_1 - (-1)^{r'}v_{(r+1,\, 0)}.
\end{split}
\end{equation}

The above calculations determine the matrix $C_2(q_1, q_2) - C_1(q_1, q_2)$, and hence $C_2(q_1. q_2)$, completely. Notice that the only appearance of the monomial $q_1$ is in \eqref{q1-app2}, with all the other entries being other $0$, $1$, $q_2$ or $q_1 q_2$ up to sign. In particular, by combining with \eqref{K-block} we conclude the proof of Lemma \ref{l:C2}.

The first calculation in \S \ref{ss:zp1} (resp.~\S \ref{ss:zp2}) shows that the first column of $C_1$ (resp.~$C_2$) has the same behavior as $h\cup$ (resp.~$\xi\cup$). Thus no mirror transform is needed and we conclude 

\begin{theorem} \label{t:main1}
For the local model of simple $(r, r')$ flip $f: X \dasharrow X'$, under the $\Tp$-corrected frame, the Drobrovin connection for $QH(X)$ in the small parameter space are determined by $A_k = C_k$ given in Lemma \ref{l:K} and Lemma \ref{l:C2}.
\end{theorem}

\subsection{The degenerate cases: simple flops and simple blow-ups.}

\subsubsection{The simple flops: $r = r'$.}

The case of flops can be considered as a degenerate case of flips. Most of the discussions for flips will be valid except that some boundary cases need to be taken care carefully.

First of all, $\T$ induces a group isomorphism $H(X) \cong H(X')$ with inverse $\T^{-1} = \Tp$, and with Poincar\'e pairing preserved. In particular $K = \ker \T = 0$. Nevertheless, the $\Tp$-corrected frame $v_\mathbf{e}$ in \eqref{e:3} is still well defined with $r' = r$ understood. In particular, using $\Box_\ell$ in \eqref{e:2} we find
\begin{equation} \label{e:f}
\begin{split}
v_{(r + 1, 0)} &= (z\p_1)^{r + 1} I + (-1)^r (z\p_2 - z\p_1)^{r + 1} I\\
&= (1 + (-1)^r q_1^{-1}) (z\p_1)^{r + 1} I\\
&= (-1)^r \f (q_1)^{-1} K_1
\end{split}
\end{equation}
where the fundamental rational function
$$
\f (q) = \frac{q}{1 - (-1)^{r  + 1} q}
$$ 
appears naturally and $K_1 = (z\p_1)^{r' + 1} I$ is defined in \eqref{k:1}. The point is that, when $r > r'$ the kernel of $\T$ exists nontrivially and $K_1$ is independent of $v_{(r + 1,\,0)}$. The relation \eqref{e:f} exists only in the case of flops!

All the calculations done in \S \ref{ss:zp1} about $z\p_1 v_\mathbf{e}$ are valid except that the term $K_1$ in \eqref{e:edge} needs to be further substituted by \eqref{e:f}. Namely for the boundary case $e_1 + e_2 = r$, \eqref{e:edge} becomes
\begin{equation} \label{e:edge1}
\begin{split}
\zp_1v_{(e_1,\, e_2)} &= v_{(e_1 + 1,\, e_2)} -  (-1)^{e_2} \f (q_1)\, v_{(r + 1,\, 0)}.
\end{split}
\end{equation}

All the calculations done in \S \ref{ss:zp2} about $(z\p_2 - z\p_1) v_\mathbf{e}$ are also valid as long as we notice that, under the assumption $r = r'$, all the cases with appearance of $K_j$ actually does not exist except for the $K_1$ in \eqref{e:211}. Since $K_1$ is just treated by \eqref{e:f}, we thus conclude that 

\begin{theorem} \label{t:flops}
For the projective local model of simple $P^r$ flops $f: X \dasharrow X'$, the $\Tp$-corrected frame $\{v_\mathbf{e}\}$ on $X$ leads to the connection matrices $C_1(q_1, q_2, \f)$ and $C_2(q_1, q_2, \f)$ such that they are independent of $z$. 

Indeed they are precisely the Dubrovin connection matrices over the small parameters: $C_1 = A_1$ and $C_2 = A_2$. Moreover, all the monomials in $q_1$ and $q_2$ are either $1$, $q_2$ or $q_1 q_2$ up to sign. 
\end{theorem}

The reason that no (generalized) mirror transformation is needed comes from the simple fact that the first columns of $C_1$ and $C_2$ has the correct form as the classical product. Since $X$ and $X'$ have isomorphic Picard--Fuchs ideal for small $I = J$ function, and the above construction of $A_1$, $A_2$ depends only on the Picard--Fuchs ideal, we get the analytic continuation of Dubrovin connection, along the small parameters, under simple $P^r$ flops.    

\begin{example} [Atiyah flop]
For $r = 1$, we get the $\Tp$-corrected frame \small
\begin{equation*}
\begin{split}
v_1 &= I, \quad v_2 = \hat h I, \quad v_3 = (\hat\xi - \hat h) I, \\
v_4 &= \hat h^2 I - (\hat \xi - \hat h)^2 I, \quad v_5 = \hat h (\hat \xi - \hat h) I + (\hat \xi - \hat h)^2 I, \\
v_6 &= \hat h^2 (\hat \xi - \hat h) I + \hat h (\hat\xi - \hat h)^2 I.
\end{split}
\end{equation*} \normalsize
The relation between $v_4$ and $K_1 = (\zp_2 - \zp_1)^2 I$ is given by \small
$$
v_4 = -\f^{-1} (\zp_1)^2 I = - q_1 \f^{-1} (\zp_2 - \zp_1)^2 I,
$$ \normalsize
where $\f = \f(q_1)$. Then the connection matrices for $\zp_1$ and $\zp_2$ are  \small
\begin{equation*}
\begin{split}
A_1 = 
\begin{bmatrix} &&& q_1 q_2 \\ 1 \\ &&&&& q_1 q_2 \\ & -\f & q_1^{-1} \f \\ && 1 \\&&& -1 & 1 
\end{bmatrix}, \quad 
A_2 = 
\begin{bmatrix} &&& q_2(1 -q_1) & q_2 \\ 1 &&&&& q_2 \\ 1 &&&&& q_1 q_2\\ & 1 \\ & 1 & 1\\ &&&& 1 
\end{bmatrix}.
\end{split}
\end{equation*} \normalsize
Notice that $v_6 = \hat h \hat \xi (\hat \xi - \hat h) I = \hat h \hat \xi^2 I - q_1 q_2 I$ does not come from a naive quantization. The $z$-independence of $C_k$'s fails if $v_6$ is not $\Tp$-corrected.
\end{example}

\subsubsection{Simple blow-ups: $r' = 0$.} \label{ss:blowup} 

For the other extreme of degenerate cases $r' = 0$, we actually get a blow-up $f: X = {\rm Bl}_p X' \to X'$ at a point $p \in X$. The structure of simple blow-ups is particularly simple. Notice that now $\xi - h = [Z] \cong P^r$ is the $f$-exceptional divisor.

Our discussion of $\Tp$-corrected frame on flips is valid by simply setting $r' = e_2 = 0$ in Definition \ref{d:q-frame}. Namely $K = \ker \T$ is spanned by
\begin{equation} \label{e:kb}
K_{i+1} := \hat \kappa_i I = \hat h^i(\hat \xi - \hat h) I, \qquad i \in [0, r - 1].
\end{equation} 
And the basis elements corresponding to $\Tp H(X')$ are
\begin{equation} \label{e:kp}
\begin{split}
v_e &:= \hat h^e  I  + (1 - {\delta}_{e, 0}) \,\hat h^{e - 1} (\hat \xi - \hat h) I, \qquad e \in [0, r + 1], \\
&=\begin{cases} \hat h^0 I = I, & e = 0; \\ \hat h^{e - 1} \hat \xi \, I, & e \in [1, r + 1]. \end{cases} 
\end{split}
\end{equation}

\begin{example} [Hirzebruch surface $F_1$] \label{E:F1}
For $(r, r') = (1, 0)$, we get the Hirzebruch surface $X = F_1 = P_{P^1}(\mathscr{O}(-1) \oplus \mathscr{O})$ as the blow-up of $X' = P^2$ at a point $p \in X'$. The $\Tp$-corrected frame is
\begin{equation*}
v_0 = I, \quad v_1 = \hat \xi \, I, \quad v_2 = \hat h \hat \xi \, I, \quad K_1 = (\hat \xi - \hat h)I
\end{equation*}
with Picard--Fuchs operators $\Box_\ell = \hat h^2 - q_1 (\hat \xi - \hat h)$, $\Box_\gamma = \hat \xi(\hat \xi - \hat h) - q_2$. 

The connection matrices for $z\p_1$ and $z\p_2$ are then obtained easily:
\begin{equation} \label{e:qF1}
A_1 = \left[\begin{array} {ccc|c} && q_1 q_2 \\ 1 &&& \\ & 1&& 1 \\ \hline -1 &&& -q_1\end{array}\right], \qquad
A_2 = \left[\begin{array} {ccc|c} & q_2 & q_1 q_2 & q_2 \\ 1 && q_2 \\ & 1 && \\ \hline && -q_2 \end{array} \right].
\end{equation}
\end{example}

\section{An example on simple (2, 1) flip for fourfolds} \label{s:(2,1)}

\subsection{Dubrovin connection on $H(X)$} \label{s:ex-(2,1)}

Let
$$
\hat t = t^0 T_0 + t^1 h + t^2 \xi \in H^0(X) \oplus H^2(X)
$$
be the small parameters, $q_1 = q^\ell e^{t^1}$, and $q_2 = q^\gamma e^{t^2}$. 

With respect to the basis given in Definition \ref{d:q-frame}, namely \small
\begin{equation} \label{basis-21}
\begin{split}
v_1 &= I, \\
v_2 &= \hat h I, \quad v_3 = (\hat \xi - \hat h) I, \\
v_4 &= \hat h^2 I - (\hat \xi - \hat h)^2 I, \quad v_5 = \hat h(\hat \xi - \hat h) I + (\hat \xi - \hat h)^2 I, \\
v_6 &= \hat h^3 I - \hat h(\hat \xi - \hat h)^2 I,\quad v_7 = \hat h^2(\hat \xi - \hat h) I + \hat h(\hat \xi - \hat h)^2 I, \\
v_8 &= \hat h^3 (\hat \xi - \hat h) I + \hat h^2(\hat \xi - \hat h)^2 I, \\
v_9 &= K_1 = (\hat \xi - \hat h)^2 I,
\end{split}
\end{equation} \normalsize
we get the Dubrovin connection matrices $A_k^v$'s on $H(X)$: 

\small
\begin{equation} \label{A-21}
A_1^v(\hat t) = h*_{t = \hat t} = 
\begin{bmatrix}
&&&&& q_1 q_2 \\
1 \\
&&&&&&& q_1 q_2 \\
& 1 \\
&& 1\\
&&& 1 &&&&& -1 \\
&&&& 1 \\
&&&&& -1 & 1 \\
& 1 & -1 &&&&&& q_1
\end{bmatrix}, 
\end{equation}
\begin{equation}
A_2^v(\hat t) = \xi *_{t = \hat t} = 
\begin{bmatrix}
&&& -q_2 & q_2 & q_1 q_2 &&& q_2\\
1 &&&&& -q_2 & q_2\\
1 &&&&&&& q_1 q_2 \\
& 1 &&&&&& q_2\\
& 1 & 1\\
&&& 1 &&&&& \\
&&& 1 & 1 \\
&&&&&& 1 \\
&&&&&&& q_2 &
\end{bmatrix}. 
\end{equation}
\normalsize
Notice that $\ker \T$ is spanned by the one dimensional direction
$$
K_1 = (\zp_2 - \zp_1)^2 I.
$$ 
It is precisely the location where $A_1$ has a pole at $q_1 = \infty$, i.e.~$q_1' = 0$.

Next we study the analytic structure of the Dubrovin connection along the infinity divisor $q_1 = \infty$ on the Hopf--M\"obius stripe $\M$ and its relation to the Dubrovin connection on $H(X')$. For this purpose, we use coordinates 
$$
x := q_1' = {1}/{q_1}, \qquad y := q_2' = q_1 q_2.
$$
The chain rule implies that
$$
y\,\p_y = xy\, \p_{q_2},
$$ 
and 
$$
x\,\p_x = x(-x^{-2}\, \p_{q_1} + y\,\p_{q_2}) = \p_{\xi - h}. 
$$

Further simplifications are possible by choosing the basis to be orthogonal. Although this simplification is not strictly necessary, it will however make the structure of the connection more transparent.

Indeed, let $w_i = \sum_{j} v_j T_{ji}$ with
\begin{equation*}
T = 
\left[\begin{array} {cccccccc|c}
1 &&&&&&&&\\
& 1 &&&&&&&\\
& \frac{1}{2} & 1 &&&&&&\\
&&& 1 &&&&&\\
&&& \frac{1}{2} & 1 &&&&\\
&&&&& 1 &&&\\
&&&&& \frac{1}{2} & 1 &&\\
&&&&&&& 1 &\\ \hline
&&&&&&&& 1
\end{array} \right],
\end{equation*}
then we have the underlying topological Poincar\'e pairing on $H(X)$:
\begin{equation} \label{pPp}
g_{ij} := (w_i, w_i)^X = \delta_{9, i + j}, \qquad 1 \le i, j \le 8,
\end{equation}
and $w_9 = v_9 = K_1$ satisfies $(w_9, w_i)^X = \delta_{9, i}$. 

Under this frame $w_i$'s, we compute the QDE (here $A_k = A_k^w$):\small
\begin{equation} \label{QDE-sym}
z(x\,\p_x) S = A_1 S = 
\left[ \begin{array} {cccccccc|c}
&&& -\frac{1}{2} xy & xy &&&& xy \\
&&&&& -\frac{1}{2} xy & xy &&\\
1 &&&&& \frac{1}{4} xy & -\frac{1}{2} xy &&\\
&&&&&&& xy \\
& 1 &&&&&& -\frac{1}{2} xy\\
&&&&&&&& 1 \\
&&& 1 &&&&& -\frac{1}{2} \\
&&&&& 1 &&&\\ \hline
& -\frac{1}{2} & 1 &&&&& xy & -1/x
\end{array} \right] S, 
\end{equation}
\begin{equation} \label{QDE-2}
z(y\,\p_y) S = A_2 S = 
\left[ \begin{array} {cccccccc|c}
&&& -\frac{1}{2} xy & xy & y &&& xy \\
1 &&&&&-\frac{1}{2} xy & xy & \\
\frac{1}{2} &&&&& \frac{1}{4} xy & -\frac{1}{2} xy & y \\
& 1 &&&&&& xy \\
& 1 & 1 &&&&& -\frac{1}{2} xy \\
&&& 1 &&&&&\\
&&& 1 & 1 &&&&\\
&&&&& \frac{1}{2} & 1 &&\\ \hline
&&&&&&& xy
\end{array} \right] S.
\end{equation}
\normalsize
The symmetry pattern respects the Poincar\'e pairing \eqref{pPp} due to the Frobenius property $(T_i*T_j, T_k) = (T_j, T_i* T_k)$: for both matrices, the first $8 \times 8$ block is symmetric under $(i, j) \mapsto (9 - j, 9 - i)$. For the remaining $9$-th column and row, it is symmetric under $(9, j) \mapsto (9 - j, 9)$.

The symmetry patterns can be unified to $(i, j) \mapsto (10 - j, 10 - i)$ if we insert the basis element $K_1$ in its original natural ordering $w_5$ instead of $w_9$. The reason for moving it to the end of the matrix is merely for the convenience of presentation on the block decomposition of the Dubrovin connection we shall perform. 

Notice that the system is irregular along $x = 0$ with Poincar\'e rank one. Let $S = P Z$ be the formal gauge transformation leading to the block decomposition with respect to the basis $w_i$'s:
\begin{equation*}
\begin{split}
z (x\,\p_x) Z &= E_1\, Z, \\
z (y \,\p_y) Z &= E_2\, Z. 
\end{split}
\end{equation*}
Then a recursive algorithm with respect to the power of $x$ determines $P$ and $E_1$ as matrices in $\Bbb Q[y, z][\![x]\!]$. And then the block decomposition on $E_2$ follows automatically since the connection is integrable (flat).

The matrix $P$ has the form
$$
P(x, y, z) = 
\begin{bmatrix}
1 &&& g_1 \\
& \ddots && \vdots \\
&& 1 & g_8 \\
f_1 & \cdots & f_8 & 1
\end{bmatrix},
$$
where $P^{11} = I_8$, $P^{22} = I_1$, and we have the new (non-constant) frame
\begin{equation} \label{frame}
\begin{split}
\tilde w_i &= w_i + f_i K_1, \quad 1 \le i \le 8, \\
\tilde K_1 &= K_1 + \sum_{i = 1}^8 g_i w_i.
\end{split}
\end{equation} 

\begin{corollary} \label{c:a-sym}
For $1 \le i \le 8$,
\begin{equation} \label{ON-sym}
f_i (x, y, z) = - g_{9 - i}(x, y, -z).
\end{equation}
\end{corollary}

\begin{proof} 
Notice that $\Lb \tilde w_i, \tilde K_1 \Rb = g_{9 - i}(-z) + f_i(z)$. Also the block decomposition leads to decomposition of bundles and connections. By Lemma \ref{metrical}, 
$$
\p_j \Lb \tilde w_i, \tilde K_1 \Rb = \Lb \nabla^z_j \tilde w_i, \tilde K_1\Rb + \Lb \tilde w_i, \nabla^z_j \tilde K_1 \Rb = 0 = \Lb \tilde w_i, \tilde K_1\Rb
$$ 
at $x = 0 = y$, $j = 1, 2$. This implies $\Lb \tilde w_i, \tilde K_1 \Rb = 0$ as power series in $x, y$.
\end{proof}

The polarized Hermitian pairing under the frame $\tilde w_i$'s takes the form
$$
\tilde g_{i\bar j} = \Lb \tilde w_i, \tilde w_j \Rb = \delta_{ij} + f_i \bar g_j.
$$ 
Hence the dual frame $\tilde w^i = \sum \tilde g^{i\bar j} \overline{\tilde w_j}$ can be determined in explicit terms. For ease of notations, we denote by $i' = 9 - i$ for $i \in [1, 8]$. Hence $w^i = w_{i'}$ and
$$
f_i = -\bar g_{i'}.
$$

\begin{lemma} \label{l:dual}
We have $D := \det P = 1 - \sum_{i = 1}^8 f_i g_{i} = \Lb \tilde K_1, \tilde K_1 \Rb$, and the dual frame of $\tilde w_i$'s and $\tilde K_1$ are 
\begin{equation*}
\tilde K^1 = \frac{\overline{\tilde K_1}}{D}, \qquad \tilde w^i = w^i + \frac{f_{i'}}{D} \tilde K_1 = w_{i'} + f_{i'} \overline{\tilde K^1}. 
\end{equation*}
\end{lemma}

\begin{proof}
Only the dual frame property needs to be verified: 
\begin{equation*}
\begin{split}
\Lb \tilde w^i, \tilde K_1 \Rb &= \Lb w^i, \tilde K_1 \Rb + f_{i'} = \bar g_{i} + f_{i'} = 0, \\
\Lb \tilde K^1, \tilde w_j \Rb &= D^{-1} \Lb \tilde K_1, \tilde w_j \Rb = D^{-1} (f_{j'} + \bar g_j) = 0,\\
\Lb \tilde w^i, \tilde w_j \Rb &= \Lb w^i, \tilde w_j \Rb + f_{i'} \Lb \overline{\tilde K^1}, \tilde w_j \Rb = \delta^i_j.\\
\end{split}
\end{equation*}
Thus $\tilde w^i$'s and $\tilde K^1$ form the dual frame.
\end{proof}

Set 
$$
f_\bullet = (f_1, \ldots, f_8), \qquad g^\bullet = (g_1, \ldots, g_8)^T.
$$
From 
$$
-z \p_k P + A_k P = P E_k,
$$ 
the block decomposition is equivalent to 
\begin{equation} \label{e:APB}
\begin{split}
&\begin{bmatrix}
A_k^{11} + A_k^{12} f_\bullet & -z\p_k g^\bullet + A_k^{11} g^\bullet + A_k^{12} \\
-z\p_k f_\bullet + A_k^{21} + A_k^{22} f_\bullet & A_k^{21} g^\bullet + A_k^{22}
\end{bmatrix} \\
&\qquad = 
\begin{bmatrix}
E_k^{11} & g^\bullet E_k^{22} \\
f_\bullet E_k^{11} & E_k^{22}
\end{bmatrix}.
\end{split}
\end{equation}
Here we are using the notations $\p_1 = x\,\p_x$ and $\p_2 = y.\p_y$. 

In particular we get the equation for $f_i$:
\begin{equation} \label{eq-fi}
\begin{split}
z\p_k f_i &= A_k^{22} f_i + (A_k^{21})_i - \sum_{j = 1}^8 f_j (E_k^{11})_{ji} \\
&= -\frac{\delta_{k1}}{x}\, f_i + (A_k)_{9i} - \sum_{j = 1}^8 \Big( f_j (A_k)_{ji} + f_j (A_k)_{j9} f_i \Big).
\end{split}
\end{equation}
Since $P$, $E$ can be solved recursively, we may reinterpret \eqref{eq-fi} as an inhomogeneous perturbation of the irregular ODE $zx\p_x\,h = -x^{-1} h$. Similar observation applies to $g_i$ too.

The first few terms of the formal power series $g_i$'s are listed below: \small
\begin{equation*} 
\begin{split}
g_1 &= -x^2 y (1 + 2zx + 6x^2 z^2 + 24 z^3 x^3 + (120 z^4 + 5y)x^4 \\
&\qquad + (720 z^5 + 63 yz) x^5 + (5040 z^6 + 642 yz^2) x^6 + \cdots), \\
g_2 &= -x^3 y (1 + 4zx + 18 z^2 x^2 + 96 z^3 x^3 + (600 z^4 + 7y) x^4 \\
&\qquad + (4230 z^5 + 115 yz) x^5 + (35280 z^6 + 1448 yz^2) x^6 +\cdots), \\
g_3 &= \tfrac{1}{2} x^3 y (3 + 14 zx + 70 z^2 x^2 + 404 z^3 x^3 + (2688z^4 + 23 y) x^4 \\
&\qquad + (20376 z^5 + 407 yz) x^5 + (173808 z^6 + 5454 y z^2) x^6 +\cdots), \\
g_4 &= -x^4 y (1 + 7zx + 46 z^2 x^2 + 326 z^3 x^3 + (2556 z^4 + 9y) x^4 \\
&\qquad + (22212 z^5 + 192 yz) x^5 +\cdots), 
\end{split}
\end{equation*}
\begin{equation*}
\begin{split}
g_5 &= \tfrac{1}{2} x^4 y (3 + 23 zx + 162 z^2 x^2 + 1214 z^3 x^3 + (9972 z^4 + 29 y) x^4\\
&\qquad + (90180 z^5 + 654 yz) x^5 + \cdots), \\
g_6 &= -x(1 + zx + 2z^2 x^2 + 6z^3 x^3 + (24z^4 + 3y)x^4 + (120 z^5 + 30 yz) x^5 \\
&\qquad + (720 z^6 + 253 y z^2) x^6 + (5040 z^7 + 2168 yz^3) x^7 + \cdots), \\
g_7 &= \tfrac{1}{2} x(1 + zx + 2z^2 x^2 + 6z^3 x^3 + (24 z^4 + 5y) x^4 + (120 z^5 + 54 yz) x^5 \\
&\qquad + (720 z^6 + 489 yz^2) x^6 + (5040 z^7 + 4472 yz^3) x^7 + \cdots), \\
g_8 &= x^2(1 + 3zx + 11z^2 x^2 + 50 z^3 x^3 + (274 z^4 + 6y) x^4 + (1764 z^5 + 87 yz) x^5 \\
&\qquad + (13068 z^6 + 986 yz^2) x^6 + (109584 z^7 + 10803 yz^3) x^7 + \cdots).
\end{split}
\end{equation*}
\normalsize
All of them are indeed named special generating series. In the following we give explanations for the main sub-series $g_i^\circ$ with the lowest $y$ degree. 

The main sub-series in the three series $g_1$, $g_6$, $g_7$, denoted by $g_i^\circ(s)$ with $s = zx$, are multiple of the factorial series $g(s) = \sum_{n = 0}^\infty n! \,s^n$, and $g_2^\circ(s)$ is essentially its derivative $g'(s) = \sum_{n = 1}^\infty n \cdot n! \, s^{n - 1}$.

The coefficients of $g_8^\circ$ is known as \emph{Stirling numbers of first kind}, which counts the number of $\sigma \in S_{n + 1}$ with exactly two cycles. It satisfies $a_0 = 1$,
\begin{equation} \label{e:S1}
a_{n} = (n + 1) a_{n - 1} + n!, \qquad n \ge 2.
\end{equation}
Its closed form is simply given by
\begin{equation} \label{e:S1'}
a_n = (n + 1)! H_{n + 1},
\end{equation}
where $H_n = \sum_{k = 1}^n 1/k$ is the harmonic series. 

The coefficients of $g_4^\circ$, treated as $a_1 = 0$, $a_2 = 1$, $a_3 = 7$, etc., satisfy
$$
a_n = n! (n - H_n) 
$$
Recursively, $a_{n} = n a_{n - 1} + (n - 1) (n - 1)!$ for $n \ge 1$.

For $g_3^\circ$, we consider the series \small
\begin{equation*}
\begin{split}
g_3 + \tfrac{1}{2} g_2 &= x^3 y(1 + 5zx + 26 z^2 x^2 + 154 z^3 x^3 + 1044 z^4 x^4 + 8028 z^5 x^5 + \cdots) \\
&\qquad + \tfrac{1}{2} x^7 y^2 (9 + 177zx + 2558 z^2 x^2 + \cdots) + \cdots. 
\end{split}
\end{equation*}
\normalsize
The coefficients $1, 5, 26, \ldots$ satisfy
$$
a_n = (n + 1)! (H_{n + 1} - 1).
$$
Recursively, $a_0 = 0$ and $a_n = (n + 1) a_{n - 1} + n!$ for $n \ge 1$.

For $g_5^\circ$, similarly, we consider the series \small
\begin{equation*}
\begin{split}
g_5 + \tfrac{1}{2} g_4 &= x^4 y (1 + 8zx + 58z^2 x^2 + 444 z^3 x^3 + 3708 z^4 x^4 + 33984 z^5 x^5 + \cdots) \\
&\qquad + \tfrac{1}{2} x^8 y^2 (11 + 270 zx + \cdots) + \cdots.
\end{split}
\end{equation*}
\normalsize
The coefficients $1, 8, 58, \ldots$ satisfy
$$
a_n = (n + 2)! (H_{n + 2} - 2) + (n + 1)!. 
$$
Recursively, $a_0 = 0$ and $a_n = (n + 2) a_{n - 1} + n \cdot n!$ for $n \ge 1$.

\begin{remark}
The calculations suggest that the only essential power series in $zx$ to be considered is the factorial series $g$. All the other series arise from standard algebraic operations and/or differentiations on the exponents (Frobenius method) which produces the harmonic series naturally. 

It is easy to see that 
$$
h := x g(-zx) = \sum_{n = 0}^\infty (-1)^n n! z^n x^{n + 1}
$$ 
satisfies the irregular ODE
$$
zx h' = -\frac{1}{x} h + 1.
$$
This is the simplest inhomogeneous perturbation of the equation 
\begin{equation} \label{k-eq}
zx h_0' = - \frac{1}{x} h_0,
\end{equation}
whose solution is $h_0 = A e^{1/(z x)}$. Equation \eqref{k-eq} is precisely the equation which appears in the kernel part of \eqref{QDE-sym}, thus we call it the \emph{kernel equation}. 

All the series $f_i$'s and hence $g_i$'s are all determined through certain perturbations of the kernel equation. It is important to locate the topological or geometric data which determines the perturbation.   
\end{remark}

For quantum cohomology, the most important part of the frame $\tilde w_i = w_i + f_i K_1$ in \eqref{frame} is however its restriction to the slice $z = 0$. This can be determined in a purely algebraic manner:

\begin{theorem} \label{t:frame}
Denote $f_1(x, y, 0), \ldots, f_8(x, y, 0)$ by 
$$
x^2 h_1, \, x h_2, \, x h_3, \, h_4, \, h_5, \, x^{-1} h_6, \, x^{-1} h_7, \, x^{-2} h_8
$$ 
respectively, and let $t := x^4 y$ be the Calabi--Yau variable. Then all $h_i$'s are power series in $t$. $X = h_1(t)$ satisfies the $9$-th degree polynomial 
\begin{equation} \label{e:key-poly}
F(X) = 1 + X + 6t X^2 + 3t^2 X^3 -2t^3 X^5 + 3t^4 X^6 + t^6 X^9
\end{equation} 
with explicit analytic formula given by Lambert's generalized binomial series
\begin{equation} \label{e:det}
h_1(t) = -\mathscr{B}_9(t)^6 = -\sum_{n = 0}^\infty \frac{2}{3n + 2} {9n + 6 \choose n} t^n
\end{equation}
(cf.~\cite[\S 5.4, p.201]{Knuth}), which converges in $|t| < 8^8/9^9$. 

Moreover, all $h_j$'s, $j = 2, \ldots, 8$, can be explicitly expressed as polynomials in $h_1$ with degree $\le 8$, and with coefficients in $\Bbb Q (t)$.
\end{theorem}

\begin{proof}
We apply the second equation in \eqref{eq-fi} to the case $k = 2$ and restrict to the case $z = 0$. The derivative term disappears and we arrive at the following non-linear system
\begin{equation} \label{e:NL}
\begin{split}
h_2 + \tfrac{1}{2} h_3 + t h_1^2 &= 0, \\
h_4 + h_5 + t h_1 h_2 &= 0,\\
h_5 + t h_1 h_3 &= 0, \\
h_6 + h_7 - \tfrac{1}{2} t h_1 + t h_1 h_4 &= 0,\\
h_7 + t h_1 + t h_1 h_5 &= 0,\\
\tfrac{1}{2} h_8 + t h_1 - \tfrac{1}{2} t h_2 + \tfrac{1}{4} t h_3 + t h_1 h_6 &= 0, \\
h_8 + t h_2 - \tfrac{1}{2} t h_3 + t h_1 h_7 &= 0,\\
-1 + h_3 + h_4 - \tfrac{1}{2} h_5 + h_1 h_8 &= 0.
\end{split}
\end{equation}
By viewing $t, h_1$ as parameters, we may regard \eqref{e:NL} as a linear system in the 7 unknowns $h_2, \ldots, h_8$. Since the 8 equations in \eqref{e:NL} are consistent, we must have 
\begin{equation*}
\begin{vmatrix}
t h_1^2 & 1 & \tfrac{1}{2}\\
& t h_1 && 1 & 1 \\
&& t h_1 && 1 \\
-\frac{1}{2} t h_1 &&& t h_1 & & 1 & 1 \\
t h_1 &&&& t h_1 && 1 \\
t h_1 &-\frac{1}{2} t & \frac{1}{4} t &&& t h_1 && \tfrac{1}{2} \\
& t & -\tfrac{1}{2} t &&&& t h_1 & 1 \\
-1 && 1 & 1 & -\tfrac{1}{2} &&& h_1 
\end{vmatrix} = 0.
\end{equation*}
where the constant terms are put in the first column. It is straightforward to compute the determinant to get the polynomial \eqref{e:key-poly}:
$$
t(1 + h_1 + 6t h_1^2 + 3t^2 h_1^3 -2t^3 h_1^5 + 3t^4 h_1^6 + t^6 h_1^9) = t F(h_1).
$$

Now by Carmer's rule, all $h_2, \ldots, h_8$ can be solved in terms of rational expressions in $h_1$ (and $t$). Since $F(X)$ is irreducible, the elementary field extension theory shows that all those rational expressions can be written as polynomials in $h_1$ with degree $\le 8$, with coefficients in $\Bbb Q(t)$.

It remains to prove \eqref{e:det}. Once we know the expected expression in the generalized binomial series, the proof becomes a direct substitution as to be shown below. In reality, the expression \eqref{e:det} is found by calculating the first few terms from the recursive relations deducing from \eqref{e:key-poly} and by an internet search on integer sequences.

We start with the definition and properties of Lambert's generalized binomial series. The general reference on this is \cite[\S 5.4]{Knuth}. For any $s \in \Bbb R_{\ge 0}$,
$$
\mathscr{B}_s(t) := \sum_{n \ge 0} {sn + 1 \choose n} \frac{1}{sn + 1} t^n.
$$
Moreover, for all $l \in \Bbb R$, taking powers corresponds to \emph{twists}:
\begin{equation} \label{e:Bsl}
\mathscr{B}_s(t)^l = \sum_{n \ge 0} {sn + l \choose n} \frac{l}{sn + l} t^n.
\end{equation} 
It is then easily seen that  $\mathscr{B}_s(t)$ satisfies a simple algebraic equation:
\begin{equation} \label{e:Leq}
t \mathscr{B}_s(t)^s = \mathscr{B}_s(t) - 1.
\end{equation}
In our current situation we need only the case $s = 9$. Namely for $b := \mathscr{B}_9(t)$ we have an equation 
\begin{equation} \label{e:Leq9}
t b^9 = b - 1,
\end{equation}
and we would like to show that $h_1 := -b^6$ satisfies the equation $F(X) = 0$ in \eqref{e:key-poly}. This is then of course just a simple algebra. Indeed by substituting $X = -b^6$ in \eqref{e:key-poly} and making use of \eqref{e:Leq9}, we get
$$
-(b - 1)^6 + 3(b - 1)^4 + 2b^3(b - 1)^3 - 3(b - 1)^2 + 6b^3 (b - 1) - b^6 + 1
$$
which reduced to zero as expected. The proof is complete.
\end{proof}

\begin{proposition} \label{p:frame}
All $h_2, \ldots h_8$ are polynomials in $b := \mathscr{B}_9(t)$ with degree $\le 8$ with coefficients in $\Bbb Q[t]$. Explicit formulae are given below:
\begin{equation}
\begin{split}
h_1 &= -b^6, \\
h_2 &= \tfrac{1}{2} b^3 - b^4, \qquad h_3 = b^3, \\
h_4 &= \tfrac{1}{2}(1 + b) - b^2, \qquad h_5 = -1 + b, \\
h_6 &= -\tfrac{1}{2} b^7 t - b^8 t, \qquad h_7 = b^7 t, \\
h_8 &= b^5 t.
\end{split}
\end{equation}
\end{proposition}

\begin{proof}
The formal part follows form Theorem \ref{t:frame}, \eqref{e:Bsl} and \eqref{e:Leq9}.

The explicit formulae can be obtained by straightforward yet lengthy manipulations. The table is obtained with help from Mathematica. 
\end{proof}

\begin{remark}
The space spanned by $1, b, b^2$ is not covered by the span of $h_4$ and $h_5$. This \emph{missing direction} deserves further investigation.
\end{remark}

\begin{remark}
In terms of generalized hypergeometric series, we have
$$
b = \mathscr{B}_9(t) = F(\tfrac{1}{9}, \cdots, \tfrac{8}{9}; \tfrac{2}{8}, \cdots, \widehat{\tfrac{8}{8}}, \tfrac{9}{8} ; \tfrac{9^9}{8^8} t),
$$
and $b^l = F(\tfrac{l}{9}, \cdots; \tfrac{l + 1}{8}, \cdots ; \tfrac{9^9}{8^8} t)$ is the $(l - 1)$-th shift with $\tfrac{9}{9}$ and $\tfrac{8}{8}$ skipped.
\end{remark}

It is also possible to determine $f_i(x, y, z)$ without setting $z = 0$. The main idea is to make use of the homogeneity of the Gromov--Witten theory. The natural coordinates system is given by the weight zero variables
\begin{equation} \label{e:coor}
s = zx, \qquad t = x^4 y.
\end{equation}
Here $t$ is the Calabi--Yau variable. It follows that
\begin{equation*}
\begin{split}
t \p_t = y \p_y, \qquad s\p_s = x\p_x - 4 y\p_y.
\end{split}
\end{equation*}
The weight zero normalizations $\h_i$'s of $f_i$'s are given similarly by 
$$
x^2 \h_1, \, x\h_2, \, x\h_3, \, \h_4, \, \h_5, \, x^{-1} \h_6, \, x^{-1} \h_7, \, x^{-2} \h_8.
$$
After some elementary manipulations, the PDE system \eqref{eq-fi} in $(s, t)$ coordinates then reads as, in the $t$ direction:
\begin{equation} \label{e:DEt}
\begin{split}
t\p_t \h_1 &= -s^{-1} (\h_2 + \tfrac{1}{2} \h_3 + t \h_1^2), \\
t\p_t \h_2 &= -s^{-1} (\h_4 + \h_5 + t \h_2 \h_1), \\
t\p_t \h_3 &= -s^{-1} (\h_5 + t \h_3 \h_1),\\
t\p_t \h_4 &= -s^{-1} (-\tfrac{1}{2} t \h_1 + \h_6 + \h_7 + t \h_4 \h_1),\\
t\p_t \h_5 &= -s^{-1} (t \h_1 + \h_7 + t \h_5 \h_1),\\
t\p_t \h_6 &= -s^{-1} (t \h_1 - \tfrac{1}{2} t \h_2 + \tfrac{1}{4} t \h_3 + \tfrac{1}{2} \h_8 + t \h_6 \h_1), \\
t\p_t \h_7 &= -s^{-1} (t \h_2 - \tfrac{1}{2} t \h_3 + \h_8 + t \h_7 \h_1),\\
t\p_t \h_8 &= -s^{-1} (-t + t \h_3 + t \h_4 - \tfrac{1}{2} t \h_5 + t \h_8 \h_1);
\end{split}
\end{equation}
and in the $s$ direction: 
\begin{equation} \label{e:DEs}
\begin{split}
s\p_s \h_1 &= -2\h_1 + s^{-1}(4 \h_2 + \h_3 + \Phi \h_1),\\ 
s\p_s \h_2 &= - \h_2 + s^{-1}(-\tfrac{1}{2} + 4\h_4 +3\h_5 + \Phi \h_2),\\ 
s\p_s \h_3 &= -\h_3 + s^{-1}(1 + 4\h_5 + \Phi \h_3),\\ 
s\p_s \h_4 &= s^{-1}(-\tfrac{3}{2} t \h_1 + 4\h_6 + 3\h_7 + \Phi \h_4),\\ 
s\p_s \h_5 &= s^{-1}(3t \h_1 + 4\h_7 + \Phi \h_5),\\ 
s\p_s \h_6 &= \h_6 + s^{-1}(4t \h_1 - \tfrac{3}{2} t \h_2 + \tfrac{3}{4} t \h_3 + \h_8 + \Phi \h_6),\\ 
s\p_s \h_7 &= \h_7 + s^{-1}(3t \h_2 - \tfrac{3}{2} t \h_3 + 4 \h_8 + \Phi \h_7),\\ 
s\p_s \h_8 &= 2\h_8 + s^{-1}(-3t + 4t \h_3 + 3t \h_4 - \tfrac{3}{2} t\h_5 + \Phi \h_8),
\end{split}
\end{equation}
where $\Phi := 3t \h_1 - \h_6 + \tfrac{1}{2} \h_7 - 1$.

By induction on the degree of $s$, it is evident that \eqref{e:DEt}, with initial values $\tilde h_i(0, t) = h_i(t)$ in Proposition \ref{p:frame}, determines $\tilde h_i$'s completely. However, this does not give information on the structures. Instead, we make one more normalization with $\tilde h_i = L_i$ for $1 \le i \le 3$ and
\begin{equation}
\tilde h_i = t L_i, \qquad 4 \le i \le 8,
\end{equation}
and rewrite \eqref{e:DEs} as
\begin{equation} \label{e:DEs'}
\begin{split}
s\p_s L_1 &= -2L_1 + s^{-1}(4 L_2 + L_3 + \Phi L_1),\\ 
s\p_s L_2 &= -L_2 + s^{-1}(-\tfrac{1}{2} + 4 t L_4 + 3t L_5 + \Phi L_2),\\ 
s\p_s L_3 &= -L_3 + s^{-1}(1 + 4t L_5 + \Phi L_3),\\ 
s\p_s L_4 &= s^{-1}(-\tfrac{3}{2} L_1 + 4 L_6 + 3 L_7 + \Phi L_4),\\ 
s\p_s L_5 &= s^{-1}(3 L_1 + 4 L_7 + \Phi L_5),\\ 
s\p_s L_6 &= L_6 + s^{-1}(4 L_1 - \tfrac{3}{2} L_2 + \tfrac{3}{4} L_3 + L_8 + \Phi L_6),\\ 
s\p_s L_7 &= L_7 + s^{-1}(3 L_2 - \tfrac{3}{2} L_3 + 4 L_8 + \Phi L_7),\\ 
s\p_s L_8 &= 2 L_8 + s^{-1}(-3 + 4 L_3 + 3 L_4 - \tfrac{3}{2} L_5 + \Phi L_8),
\end{split}
\end{equation}
where now $\Phi = t (3L_1 - L_6 + \tfrac{1}{2} L_7) - 1 \equiv -1 \pmod{t}$.

The non-linear system \eqref{e:DEs'} reduces to a linear system when we set $t = 0$. In fact the resulting system is a perturbation of the irregular ODE \eqref{k-eq}, and it is elementary to check that the the solution is given by those special power series related to $g_i^\circ(x, 0, -z)$ written down before.

\subsection{Dubrovin connection on $H(X')$}

Under the frame \eqref{frame}, the connection matrix $E_i$, $i = 1, 2$ are decomposed into two diagonal blocks 
$$
E_i^{11} = A_i^{11} + A_2^{21} f_\bullet
$$ 
and $E_i^{22}$. From \eqref{e:APB} and \eqref{QDE-sym}, they are determined by $A_i$ and $P$ as \small
\begin{equation} \label{e:B11}
E_1^{11} = 
\begin{bmatrix}
xy f_1 & xy f_2 & xy f_3 & xy(-\frac{1}{2} + f_4) & xy(1 + f_5) & xy f_6 & xy f_7 & xy f_8 \\
&&&&& -\frac{1}{2} xy & xy \\
1 &&&&& \frac{1}{4} xy & -\frac{1}{2} xy\\
&&&&&&& xy \\
& 1 &&&&&& -\frac{1}{2} xy\\
f_1 & f_2 & f_3 & f_4 & f_5 & f_6 & f_7 & f_8 \\
-\frac{1}{2} f_1 & -\frac{1}{2} f_2 & - \frac{1}{2} f_3 & 1 - \frac{1}{2} f_4 & -\frac{1}{2} f_5 & -\frac{1}{2} f_6 & -\frac{1}{2} f_7 & -\frac{1}{2} f_8 \\
&&&&& 1
\end{bmatrix},
\end{equation}
\begin{equation} \label{B1-22}
\begin{split}
E_1^{22} &= -\frac{1}{x} + \Big( -\frac{1}{2} g_2 + g_3 + xy\, g_8\Big) \\
&= -\frac{1}{x}\Big(1 - yx^4(3 + 12 zx + 55 z^2 x^2 + 300 z^3 x^3 \\
&\qquad \qquad \qquad + 1918 z^4 x^4 + 14112 z^5 x^5 + \cdots)\\
&\qquad \qquad -y^2 x^8(21 + 348 z x + \cdots) + \cdots \Big).
\end{split}
\end{equation} 
\normalsize
Notice that the $i$-th row is affected by $f_\bullet = (f_1, \ldots, f_8)$ with a multiple given by the $(i, 9)$-th entry of $z(x\,\p_x)$ in \eqref{QDE-sym}. 

Exactly the same pattern applies to $E_2^{11} = A_2^{11} + A_2^{21} f_\bullet$ too, where now only the first row is added by $xy f_\bullet$:
\small
\begin{equation} \label{e:B211}
E_2^{11} = 
\begin{bmatrix}
xy f_1 & xy f_2 & xy f_3 & xy(-\frac{1}{2} + f_4) & xy(1 + f_5) & y + xy f_6 & xy f_7 & xy f_8 \\
1 &&&&& -\frac{1}{2} xy & xy \\
\frac{1}{2} &&&&& \frac{1}{4} xy & -\frac{1}{2} xy & y\\
& 1 &&&&&& xy \\
& 1 & 1 &&&&& -\frac{1}{2} xy\\
& & & 1 \\
&&& 1 & 1 \\
&&&&& \frac{1}{2} & 1
\end{bmatrix},
\end{equation}
\begin{equation} \label{B1-23}
E_2^{22} = \frac{1}{x} \Big(yx^4 (1 + 3zx + 11 z^2 x^2 + \cdots)\Big) = yx \,g_8.
\end{equation}
\normalsize

Now the Dubrovin connection on $H(X')$ follows from the BF/GMT procedure applied to $E_1^{11}$ and $E_2^{11}$. Namely to solve $B$ with
\begin{equation} \label{e:BF0}
z\p_a B = B E_a^{11} - B_0 E_{a; 0}^{11} B^{-1}_0 B, \qquad a = 1, 2,
\end{equation}
where $0$ means its value at $z = 0$.

This is still complicated. But in contrast to the direct computation based on the Picard--Fuchs equations on the $X'$ side, we have a better structure on the connection matrices before the BF/GMT. This allows us to determine the BF/GMT in explicit terms at least for the extremal ray directions, which is given below. 

We denote $\delta = zx\, \p_x$ and define its (pseudo) inverse $\mathscr{I}$ by 
\begin{equation} \label{e:pinv}
\mathscr{I} \phi = \mathscr{I}(\phi - \phi(z = 0)) = \int \frac{\phi - \phi(z = 0)}{zx}\,dx.
\end{equation}
For example, in terms of formal series expansion, we have
\begin{equation*}
\begin{split}
\mathscr{I} f_1 &= \tfrac{3}{3} x^3 - \tfrac{11}{4} zx^4 + \tfrac{50}{5} z^2 x^5 + \cdots \pmod{y}, \\
\mathscr{I} (f_1/x) &= \tfrac{3}{2} x^2 - \tfrac{11}{3} zx^3 + \tfrac{50}{4} z^2 x^4 + \cdots \pmod{y}.
\end{split}
\end{equation*}

\begin{lemma} \label{l:extr}
The Birkhoff factorization matrix $B$ modulo $y$ is given by \small
\begin{equation} \label{e:B-mod-y}
\begin{bmatrix}
1 \\
& 1 \\
&& 1 \\
&&& 1 \\
&&&& 1 \\
-\mathscr{I}^2 (f_1/x) & \mathscr{I} f_2 & \mathscr{I} f_3 &&& 1 \\
\frac{1}{2} \mathscr{I}^2 (f_1/x) & -\frac{1}{2} \mathscr{I} f_2 & -\frac{1}{2}\mathscr{I} f_3 &&&& 1 \\
\mathscr{I}^3 (f_1/x - f_3) & -\mathscr{I}^2 f_2 & -\mathscr{I}^2 f_3 &&&&& 1
\end{bmatrix}.
\end{equation} \normalsize
In particular, by writing $B = I + N$ we then have $N^2 = 0$ and $B^{-1} = I - N$.
\end{lemma}

\begin{proof}
We will show that its enough to onsider the ansatz 
$$
B = I + N = I + \sum_{i = 6}^8 \sum_{i = 1}^3 N_{ij} e_{ij}.
$$ 
It is clear that $N^2 = 0$ and $B^{-1} = I - N$. We need to solve $N_{ij}$ such that
\begin{equation} \label{e:zfree}
-z( x\,\p_x B) B^{-1} + B E_1^{11} B^{-1} \pmod{y}
\end{equation}
is independent of $z$. From \eqref{e:B11}, only $f_1, f_2, f_3$ in the 6-th and 7-th rows have non-constant contributions. By expanding out \eqref{e:zfree}, the constant entries remain the same as $A_1^{11}$ which are all $1$'s in the $(3, 1)$, $(5, 2)$, $(7, 4)$ and $(8, 6)$ entries. The non-constant entries are in the $3 \times 3$ block as in $(N_{ij})$ as
\begin{equation}
\begin{bmatrix}
-\delta N_{61} + N_{63} + f_1 & -\delta N_{62} + f_2 & -\delta N_{63} + f_3 \\
-\delta N_{71} + N_{73} - \frac{1}{2} f_1 & -\delta N_{72} - \frac{1}{2} f_2 & -\delta N_{73} - \frac{1}{2} f_3 \\
-\delta N_{81} + N_{83} - N_{61} & -\delta N_{82} - N_{62} & -\delta N_{83} - N_{63}  
\end{bmatrix}.
\end{equation}
The $z$ entries from $f_i$'s are then removed by setting 
$$
N_{62} = \mathscr{I} f_2, \quad N_{63} = \mathscr{I} f_3, \quad N_{72} = -\tfrac{1}{2} \mathscr{I} f_2, \quad N_{73} = -\tfrac{1}{2} \mathscr{I} f_3.
$$ 
We may then solve $N_{82} = -\mathscr{I} N_{62} = -\mathscr{I}^2 f_2$ and $N_{83} = -\mathscr{I} N_{63} = -\mathscr{I}^2 f_3$.

We also have 
$$
N_{61} = \mathscr{I} N_{63} + \mathscr{I} f_1 = \mathscr{I}^2 f_3 + \mathscr{I} f_1 = -\mathscr{I}^2 (f_1/x),
$$
where the last equality follows from \eqref{e:S1} or \eqref{e:S1'}. 

Similarly $N_{71} = \frac{1}{2} \mathscr{I}^2 (f_1/x)$. Finally, 
$$
N_{81} = \mathscr{I} N_{83} - \mathscr{I} N_{61} = \mathscr{I}^3(f_1/x - f_3)
$$ 
as expected.
\end{proof}

\begin{corollary} \label{c:GMT}
For local $(2, 1)$ flips, the Dubrovin connection matrices modulo $y$ and up to GMT are given by \small
\begin{equation}
\bar C'_1 = \begin{bmatrix}
0 \\ 0 \\ 1 \\ & 0 \\ & 1 \\
-{3 x^2}/{2} & -{x}/{2} & x \\
{3 x^2}/{4} & {x}/{4} & -{x}/{2} & 1 & 0\\
-{13 x^3}/{9} & -{x^2}/{4} & {x^2}/{2} &&& 1 & 0 & 0
\end{bmatrix}
\end{equation} \normalsize
and $\bar C'_2 = A_2^{11} \pmod{y}$. $\bar C'_1$ determines the GMT in the extremal ray variable as
\begin{equation*}
\begin{split}
&\sigma(s^1 h' + s^2 \xi') \\
&= s^1 h' + s^2 \xi' + \tfrac{3}{4} e^{2s^1} q^{2\ell'} \xi'^2 h' - \tfrac{13}{27} e^{3s^1} q^{3\ell'} \xi'^3 h' \pmod{q^{\gamma'}}.
\end{split}
\end{equation*}
\end{corollary}

Now we give a simple example where the BF/GMT can be ignored completely, namely the case of simple blow-ups studied in \S \ref{ss:blowup}

\begin{example} [Example \ref{E:F1} continued]
By repeating and specializing the discussion to $(1, 0)$ flips, we are required to block-diagonalize $A_x = A_2 - A_1$ and $A_y = A_2$ with $x := 1/q_1$ and $y = q_2' = q^{\gamma'} e^{t_2'}= q_1 q_2$, where $\gamma'$ is the line class of $X' = P^2$. Here we avoid the notation $q_1'$ or otherwise it should be $q^{\ell'} e^{t_1'}$ for the line class $\ell'$ of $Z' = P^0$! In practice the divisor $D_\infty = (x = 0)$ in the Hopf--M\"obius stripe $\M$ is the K\"ahler moduli of $X'$ with parameter $y$. 

From \eqref{e:qF1} we have
\begin{equation} \label{e:qF1'}
A_x = \left[\begin{array} {ccc|c} & xy & & xy \\ && xy & \\ &&& -1 \\ \hline 1 && -xy & -1/x \end{array}\right], \qquad
A_y = \left[\begin{array} {ccc|c} & xy & y & xy \\ 1 && xy \\ & 1 && \\ \hline && -xy \end{array} \right].
\end{equation}
In the diagonalization process all the formal series $f_\bullet$ and $g^\bullet$ in $x$ do not have constant terms. For the resulting $3 \times 3$ matrices $E_x^{11}$ and $E_y^{11}$, the BF matrix $B$ also reduces to $I_3$ modulo $x$. Thus after substituting $x = 0$ the resulting matrices go to ${\bf 0}_{3}$ and 
$$
A_{\xi'} = \left[\begin{array} {ccc} && y \\ 1 \\ & 1 \end{array} \right],
$$
which recovers the Dubrovin connection matrix on $P^2$ with $y = q^{\gamma'} e^{t'}$.

This property holds for all global blow-ups at points, a well known fact from the degeneration formula. From the point of view of Dubrovin connections, the block-diagonalization needed will follow from the more general case of simple flips. In fact it could be generalized to more general blow-ups along smooth centers \cite{LLW4} which could not be handled directly by the degenerate formula.
\end{example}

\subsection{Quantum invariance along the extremal variable}

Consider the local $(2, 1)$ flip. For $a \in H(X)$, we denote by $[a] = \tilde a(0)$ the $z$-constant part of the deformed vector under the block diagonalization. 

The procedure \eqref{e:APB} gives, for $k = 1, 2$, 
\begin{equation} \label{e:AzB}
A_k^{11} + A_k^{12} f_\bullet = E_k^{11}.
\end{equation}
Let $\hat T_i I = w_i$. The $(i, j)$-th component of the left hand side gives
$$
\langle T_k, T^i, T_j \rangle^X + \langle T_k, T^i, \kappa_0 \rangle^X f_j = \langle T_k, T^i, \widetilde T_j \rangle^X.
$$

\begin{remark}
Under gauge transform, the $z = 0$ part of the LHS in \eqref{e:AzB} should be $\langle T_k, \tilde T^i(0), \tilde T_j(0) \rangle^X$. By Lemma \ref{l:dual} and the block diagonalization, the $\tilde \kappa_0(0)$ part in $\tilde T^i(0)$ has no contribution in the quantum product. The simple effect on $z$ in \eqref{e:AzB} is due to the PDE on $f_i$'s in \eqref{eq-fi}.
\end{remark}

If no more BF/GMT is needed on $E_k^{11}$, which is the case only if we modulo $(x^2, y) $ (by \eqref{e:B11} and $f_1 = -x^2 + \cdots$), then its $(i, j)$-th component equals
$$
\langle T_k', T'^i, T'_j \rangle^{X'} 
$$
where $T_j = \Tp T'_j$. The only non-trivial case says that 
\begin{equation} \label{e:mod-y}
\langle \xi - h, \xi - h, \widetilde{\xi} - \widetilde{h}\rangle^X \equiv \langle h', h' , h' \rangle^{X'} \pmod{x^2, y}.
\end{equation}
We will see that this follows easily by a direct comparison.

From the previous calculations, especially \eqref{ON-sym} and the table on $g_i(x, y, z)$'s following it, we know that $\Tp h' = \xi - h$ and 
\begin{equation} \label{e:w3-ltft}
\begin{split}
[\xi - h] &= \tilde w_3(0) = (\xi - h) + f_3(x, y, 0)\, K_1 \\
&= (\xi - h) + x\,\kappa_0.
\end{split}
\end{equation}
Then the LHS of \eqref{e:mod-y} is simply a topological term
$$
x \langle \xi - h, \xi - h, \kappa_0 \rangle^X_{\beta = 0} = x.
$$
And the only extremal ray invariant of the RHS of \eqref{e:mod-y} is easily seen, from the $I$ function of $X'$, to be $1 \cdot q^{\ell'} = x$. Hence \eqref{e:mod-y} holds.

\begin{theorem} [Linear invariance along extremal rays] \label{t:flip-extr}
For extremal primary Gromov--Witten invariants of at least $n \ge 3$ insertions, we have
$$
\langle [\xi - h]^{\otimes n}\rangle^X = \langle (h')^{\otimes n}\rangle^{X'}.
$$
In fact this is equivalent to the quantum interpretation of Cayley's formula
$$
a_d := \langle \kappa_0^{\otimes (d + 1)}\rangle^X_{d\ell} = d^{d - 2}, \qquad d \ge 1.
$$
\end{theorem}

\begin{proof}
We start with the one point invariant. For $d_2' = 0$ we have
\begin{equation} 
\label{e:IX'}
I^{X'}_{d' \ell'} = (-1)^{3(d' - 1)}\frac{(\xi' - h')^3 \prod_{m = 1}^{d' - 1}{(h' + mz)}}{(h' + d'z)^2},
\end{equation}
while the virtual dimension is $-d' + (4 - 3) + 1 = 2 - d'$. Thus $d' = 1$ and we have a divisor insertion. Indeed from the above $I$ expression $I = J$ for $\beta' = \ell'$ and we have $\langle h' \rangle^{X"} = q^{\ell'} \, I^{X'}_{\ell'}.h' = q^{\ell'}$. Hence for all $n \ge 0$:
$$
\langle (h')^{\otimes n} \rangle^{X'} = q^{\ell'} = x.
$$ 

On the $X$ side, we compute the corresponding terms. The virtual dimension for $\beta = d\ell$ is now $d + 1 + n$. So there are exactly $d + 1$ insertions which support the class $\kappa_0$. In the following we assume that $n \ge 2$. 

We expand the homogeneous expression
\begin{equation} \label{e:homo-exp}
\begin{split}
\langle [\xi - h]^{\otimes n}\rangle^X &= \langle ((\xi - h) + x\,\kappa_0)^{\otimes n}\rangle^X \\
&= \sum_{j = 2}^n C^n_j\,x^j\,q^{(j - 1)\ell} \,\langle (\xi - h)^{\otimes n - j} \otimes \kappa_0^{\otimes j}\rangle^X_{(j - 1)\ell} + 3x\,\delta_{n, 3}\\
&= x \sum_{j = 2}^n (-1)^{n - j} C^n_j (j - 1)^{n - j} \langle \kappa_0^{\otimes j}\rangle^X_{(j - 1)\ell} + 3x\,\delta_{n, 3}
\end{split}
\end{equation}
where the divisor axiom is used, and in the case $n = 3$ there are 3 more terms coming from the classical product $(\xi - h, \xi - h, \kappa_0) = 1$. 

From \eqref{e:homo-exp}, the theorem amounts to the assertion that 
\begin{equation} \label{e:recur-a}
\sum_{j = 0}^n (-1)^{n - j} C^n_j (j - 1)^{n - j} a_{j - 1} = -3\,\delta_{n, 3}.
\end{equation}
Here we use the convention $a_{-1} = -1 = (-1)^{-1 - 2}$ and $a_0$ can be assigned arbitrarily since it always comes with the coefficient $0$.

For $n = 2$, \eqref{e:recur-a} requires that $a_1 = 1$. This can be proved by direct divisorial reconstruction or by looking at the 3 point invariant $\langle h, \kappa_0, \kappa_0 \rangle = q_1$ as shown in the explicit calculation \eqref{A-21}.

For $n = 3$, \eqref{e:recur-a} is then equivalent to $-1 + 0 - 3a_1 + a_2 = -3$, that is $a_2 = 1$. For $n = 4$, \eqref{e:recur-a} is then equivalent to 
$$
-1 - 0 + 6a_1 - 4\times 2^1 a_2 + a_3 = 0,
$$
that is $a_3 = 3$. For $n = 5$ this gives $a_4 = 16 = 4^2$. For $n = 6$ this gives $a_5 = 125 = 5^3$. Thus it is tempted to guess if $a_d = d^{d - 2}$ holds? We will see that this is indeed the case.

Recall that the Striling number of second kind $S(m, n)$ for two integers $m, n \ge 0$ is the number of partitions of $m$ elements into $n$ disjoint subsets. It is defined to be zero if $m < n$. It admits a nice relation with the combinatorial number $C^n_j$, namely for any $m, n \ge 0$, 
\begin{equation} \label{e:Stirling2}
n! S(m, n) = \sum_{j = 0}^n  (-1)^{n - j} C^n_j j^m.
\end{equation}
(cf.~\cite[p.265 (6.19)]{Knuth}). It follows from \eqref{e:Stirling2} easily that
\begin{equation} \label{e:Stirling}
\sum_{j = 0}^n (-1)^{n - j} C^n_j (j - 1)^{n - 3} = 0.
\end{equation}

Now we may continue the proof of Theorem \ref{t:flip-extr} to establish $a_d = d^{d - 2}$ under the assumption on quantum invariance relation \eqref{e:recur-a}. Indeed, by induction on $n = d + 1$, the validity of \eqref{e:recur-a} is equivalent to equation \eqref{e:Stirling} for all $n \ge 4$ by substituting $a_{j - 1} = (j - 1)^{j - 3}$ into it and notice that the power $n - 3$ is then uniform for all $j$.
 
Conversely, we will now prove the quantum invariance \eqref{e:recur-a} by establishing $a_d = d^{d - 2}$, for all $d \ge 1$ directly.

Let $a_0 = 0$, $a_1 = 1$ and $d \ge 2$. Since $\kappa_0 = (\xi - h)^2$, by applying the divisorial reconstruction we get the following recursive formula 
\begin{equation} \label{e:div}
a_d = \sum_{d' = 1}^{d - 1} d' \, a_{d'} \,a_{d - d'}\, C^{d - 2}_{d' - 1} .
\end{equation}
While it is possible to show that \eqref{e:div} is equivalent to \eqref{e:recur-a}, which must be the case after the theorem is proved, we will proceed differently. Since 
$$
d(d - 1) \frac{d'\,C^{d - 2}_{d' - 1}}{d!} = \frac{(d')^2}{(d')!} \frac{d - d'}{(d - d')!},
$$
the exponential generating function $g = \sum_{d \ge 0} a_d t^d/d!$ then satisfies 
$$
t^2 g'' = t(tg')' \cdot (tg'). 
$$
Let $\mathscr{E} = g'$. Then $\mathscr{E}' = (t\mathscr{E})' \mathscr{E}$. That is, $\mathscr{E}'/\mathscr{E} = (t\mathscr{E})'$. Since $\mathscr{E}(0) = g'(0) = 1$, after integration we get $\log \mathscr{E} = t \mathscr{E}$. That is we arrive at the famous functional equation of Euler (see e.g.~\cite[p.369]{Knuth}):
\begin{equation} \label{e:Cayley-fe}
\mathscr{E} = e^{t \mathscr{E}}.
\end{equation}
Equation \eqref{e:Cayley-fe} has explicit solution given by (cf.~\cite[p.369 (7.85)]{Knuth})
$$
\mathscr{E}(t) = \sum_{d \ge 1} d^{d - 2}\frac{t^{d - 1}}{(d - 1)!},
$$
hence $a_d = d^{d - 2}$ as expected. The proof is completed. 
\end{proof}

\begin{remark}
The number $a_d = d^{d - 2}$ is traditionally known as the Cayley number in combinatorics. It is the number of \emph{spanning trees} in the complete graph on $d$ vertexes (and hence with $d - 1$ edges). It is interesting to see if the localization techniques in evaluating $\langle \kappa_0^{\otimes (d + 1)}\rangle_{d\ell}$ leads to the graph sum corresponding to these spanning trees. 
\end{remark}

\begin{remark}
Theorem \ref{t:flip-extr} implies that for $(2, 1)$-flips, the embedding 
$$
QH(X') \hookrightarrow QH(X)
$$
is linear when restricting to the extremal ray variable. However, a lengthy yet straightforward calculation shows that
$$
\langle \tilde K_1, \tilde w_3, \tilde w^6, \tilde w_3\rangle \equiv - x^4 y \pmod{y^2},
$$
which implies that the embedding 
is necessarily non-linear.  It is an interesting question whether the embedding over the extremal ray variable is always linear for $(r, r')$-flips.
\end{remark}


\begin{thebibliography}{99}

\bibitem{CG} T.\ Coates and A.\ Givental; 
{\it Quantum Riemann--Roch, Lefschetz and Serre}, 
Ann.\ of Math.\ {\bf 165} (2007), no.1, 15--53.

\bibitem{Du} B.\ Dubrovin;
\textit{Geometry of 2D topological field theories}, in Integrable Systems and Quantum Groups, Lecture Notes in Math.\ {\bf 1620}, Springer-Verlag 1996, 120--348. 

\bibitem{Guest} M.\ Guest;
\textit{From Quantum Cohomology to Integrable Systems}, Oxford Univ.\ Press 2008.

\bibitem{Knuth} R.\ Graham, D.\ Kunth and O.\ Patashnik;
\textit{Concrete Mathematics: A Foundation of Computer Science, 2nd edition}, Addison-Wesley 1994.

\bibitem{Hertling} C.\ Hertling;
\textit{Frobenius manifolds and moduli spaces for singularities}, Cambridge University Press 2002. 

\bibitem{Iritani} H.\ Iritani;
\textit{Quantum $D$-modules and generalized mirror transformations}, Topology {\bf 47} (2008), 225--276.

\bibitem{hI2} ---{}---;
\textit{Global mirrors and discrepant transformations for toric Deligne-Mumford stacks}, SIGMA Symmetry Integrability Geom.\ Methods Appl.\ 16 (2020), Paper No.\ 032, 111 pp.

\bibitem{LLW1} Y.-P.~Lee, H.-W.~Lin and C.-L.~Wang;
\textit{Flops, motives and invariance of quantum rings}, Ann.\ of Math.\ (2) {\bf 172} (2010), no.\ 1, 243--290.

\bibitem{LLW-II} ---{}---;
\textit{Invariance of quantum rings under ordinary flops II: A quantum Leray--Hirsch theorem}, Algebraic Geometry {\bf 3} (2016),
no.\ 5, 215--253.
\bibitem{LLW4} ---{}---;
\textit{A blowup formula in Gromov--Witten theory}, work in progress.

%

\bibitem{Sibuya} Y.\ Sibuya;
\textit{Linear Differential Equations in the Complex Domain: Problem of Analytic Continuation}, AMS Transl.\ Math.\ Monograph {\bf 82}, 1990.  

\bibitem{Wasow} W.\ Wasow;
\textit{Asymptotic Expansions for Ordinary Differential Equations}, Interscience Publishers, New York, 1965.

\end{thebibliography}
\end{document}